\documentclass[11 pt, final]{article}
\title{Lipschitz 1-connectedness for some solvable Lie groups.}
\author{David Bruce Cohen}

\usepackage[latin1]{inputenc}
\usepackage{tikz}
\usetikzlibrary{shapes,arrows}
\usepackage{graphicx}
\usepackage{amsmath}
\usepackage{amssymb}
\usepackage{amsthm}
\usepackage{epsfig,pinlabel}
\usepackage{amsfonts}
\usepackage{amscd}
\usepackage{mathabx}

\DeclareMathOperator{\NN}{\ensuremath{\mathbb{N}}}

\DeclareMathOperator{\RR}{\ensuremath{\mathbb{R}}}
\DeclareMathOperator{\EE}{\ensuremath{\mathbb{E}}}

\DeclareMathOperator{\Span}{Span}
\DeclareMathOperator{\Kill}{Kill}
\DeclareMathOperator{\image}{image}
\DeclareMathOperator{\Lip}{Lip}
\DeclareMathOperator{\Hom}{Hom}

\DeclareMathOperator{\GL}{GL}
\DeclareMathOperator{\Ad}{Ad}

\def\To{{\rightarrow}}
\def\mfu{{\mathfrak{u}}}
\def\cS{{\mathcal{S}}}

\def\cC{{\mathcal{C}}}
\def\cP{{\mathcal{P}}}
\def\cW{{\mathcal{W}}}
\def\cD{{\mathcal{D}}}
\def\tilg{{\tilde{g}}}
\def\tilx{{\tilde{x}}}
\def\tilr{{\tilde{r}}}
\def\tilu{{\tilde{u}}}
\def\olx{{\overline{x}}}

\def\olu{{\overline{u}}}

\def\hU{{\hat{U}}}

\def\mfu{{\mathfrak{u}}}
\def\mfua{{\mathfrak{u}_{\alpha}}}
\def\mfub{{\mathfrak{u}_{\beta}}}
\def\la{{\langle}}
\def\ra{{\rangle}}

\def\Deps{{D^2_\epsilon}}

\theoremstyle{plain}
\newtheorem{theorem}{Theorem}[section]
\newtheorem{lemma}[theorem]{Lemma}

\newtheorem{proposition}[theorem]{Proposition}
\newtheorem{corollary}[theorem]{Corollary}
\theoremstyle{definition}
\newtheorem*{definition}{Definition}

\begin{document}

\maketitle

\begin{abstract}
A space $X$ is said to be Lipschitz 1-connected if every Lipschitz loop $\gamma:S^1\rightarrow X$ bounds a $O(\text{Lip}(\gamma))$-Lipschitz disk $f:D^2\rightarrow X$. A Lipschitz 1-connected space admits a quadratic isoperimetric inequality, but it is unknown whether the converse is true. Cornulier and Tessera showed that certain solvable Lie groups have quadratic isoperimetric inequalities, and we extend their result to show that these groups are Lipschitz 1-connected.
\end{abstract}

\section{Introduction.}
\paragraph{Lipschitz 1-connectedness.} Consider a complete, simply connected Riemannian manifold $X$ which is homogeneous in the sense that some Lie group acts on $X$ transitively by isometries. It is well known that in such a manifold, every loop admits a Lipschitz filling (see Proposition \ref{proposition:fillingspan}; for a much stronger result see \cite[Corollary 1.4]{my}). A classical way to study the large scale geometry of $X$ is by asking how hard it is to fill loops in $X$.

\begin{definition}
Let $\gamma:S^1\To X$ be a loop. The filling span of $\gamma$ denoted $\Span(\gamma)$, is defined to be
$$\inf\{\Lip(f)|f:D^2\To X;f|_{S^1}=\gamma\}.$$

We say that $X$ is Lipschitz 1-connected if there exists a constant $C$ such that $\Span(\gamma)\leq C\Lip(\gamma)$ for all Lipschitz loops $\gamma:S^1\To X$.
\end{definition}

For instance, Euclidean $n$-dimensional space $\EE^n$ is Lipschitz $1$-connected because one may ``cone off" loops in $\EE^n$: given a loop $\gamma:S^1\To X$ with $\gamma(1)=0$, we may take $f(re^{i\theta})=r\gamma(e^{i\theta})$ to obtain a $1$-Lipschitz filling. A similar argument shows that all CAT(0) manifolds are Lipschitz 1-connected.

\paragraph{Solvable groups.} In this paper, we will be interested in the case where $X$ is some solvable real Lie group $G$ equipped with a left invariant metric. To motivate this, note that Lipschitz 1-connectedness is a QI-invariant (Proposition \ref{proposition:QI}), and that every Lie group is quasi isometric to a solvable Lie group. We will further assume that $G$ has the form $U\rtimes A$, where the following conditions hold.
\begin{itemize}
\item $A$ is abelian Lie group. 
\item $U$ is a closed subgroup of the group of real $n\times n$ upper triangular matrices with diagonal entries equal to $1$ (for some $n$).
\item $A$ and $U$ are contractible.
\end{itemize}
We will now see some examples of groups which are not Lipschitz 1-connected.

\paragraph{Groups of Sol type.} Fix real numbers $t_2 > 1 > t_1 >0$ and consider the group $G$ of matrices of the form
\[\begin{bmatrix}t_1^s & 0 & x \\ 0 & t_2^s & y \\ 0 & 0 & 1\end{bmatrix}\]

where $s,x,y\in \RR$. Note that $G$ decomposes as $U\rtimes A$ as above if we take $A$ to be the diagonal matrices of $G$ and $U$ to be the matrices with diagonal entries equal to $1$. It is known \cite{wp}[Theorem 8.1.3] that there exists loops $\gamma$ in $G$ such that the minimal area of any filling of $\gamma$ is on the order of $\exp(\Lip(\gamma))$. Hence, $G$ is not Lipschitz 1-connected. Groups of this form are called groups of Sol type.

More generally, Cornulier and Tessera have shown that if a group $G=U\rtimes A$ as above surjects onto a group of Sol type, then it has loops of exponentially large area, and hence cannot be Lipschitz 1-connected \cite[Theorem 12.C.1]{ct}. If $G$ surjects onto a group of Sol type we say that $G$ has the SOL obstruction.

\paragraph{Tame groups.} On the other hand, if the conjugation action of some $a\in A$ contracts $U$, then $G=U\rtimes A$ will be Lipschitz 1-connected (Proposition \ref{proposition:tamelip1c}). For $a$ to be a contraction means that there is some compact subset $\Omega$ of $U$ such that for any compact subset $K$ of $U$, some positive power of $a$ conjugates $K$ into $\Omega$. If such $a$ and $\Omega$ exist, $G$ is said to be tame.

\paragraph{The theorem of Cornulier and Tessera.} It is clear that a space which is Lipschitz 1-connected admits a quadratic isoperimetric inequality---i.e., any loop of length $\ell$ must bound a disk of area $O(\ell^2)$. Cornulier and Tessera have given a large class of solvable Lie groups admitting quadratic isoperimetric inequalities, and we shall extend their result to show that these groups are Lipschitz 1-connected.

Given an action of the abelian group $A$ on a vector space $V$, let $V_0\subset V$ be the subspace consisting of vectors $v$ such that $\frac{\log \|a^n v\|}{n}\To 0$ as $n\To \infty$ for all $a\in A$. Their theorem \cite[Theorem F]{ct} states that $G=U\rtimes A$ satisfies a quadratic isoperimetric inequality if the following conditions hold.
\begin{itemize}
\item $(U/[U,U])_0=0$. ($G$ is said to be standard solvable if this condition holds.)
\item $G$ does not surject onto a group of Sol type.
\item $H_2(\mfu)_0=0$, where $\mfu$ is the Lie algebra of $U$ and $H_2$ denotes second Lie algebra homology---see Definition \ref{definition:killing}.
\item $\Kill(\mfu)_0=0$, where the killing module $\Kill(\mfu)$ is the quotient of the symmetric square $\mfu\odot\mfu$ of $\mfu$ by the $A$-subrepresentation spanned elements of the form $[x,y]\odot z-y\odot[x,z]$.
\end{itemize}

\paragraph{Main Theorem.} Our primary objective in this paper is to establish the following theorem, proved as Theorem \ref{theorem:MAIN}, improving Cornulier and Tessera's result to show Lipschitz 1-connectedness. 
\begin{theorem}
\label{theorem:mainintro}
Let $G$ be a group of the form $U\rtimes A$ where $U$ and $A$ are contractible real Lie groups, $A$ is abelian, and $U$ is a real unipotent group (i.e., a closed group of strictly upper triangular real matrices.)

If $(U/[U,U])_0$, $H_2(\mfu)_0$ and $\Kill(\mfu)_0$ are all trivial and $G$ does not surject onto a group of Sol type, then $G$ is Lipschitz $1$-connected.
\end{theorem}

\paragraph{Quadratic isoperimetric inequality versus Lipschitz 1-connectedness.} As noted above, if $X$ is Lipschitz 1-connected, then it has a quadratic isoperimetric inequality. It is not known whether the converse is true---that is, there are no known examples where $X$ has a quadratic isoperimetric inequality but is not Lipschitz 1-connected.  Recently, Lytchak, Wenger, and Young\cite{ryw} have shown some results about the existence of Holder fillings in spaces admitting quadratic isoperimetric inequalities. 

\subsection{Organization.}
\label{subsection:organization}
This paper is organized as follows. \S\ref{section:preliminaries} recalls some known results about Lipschitz filling in homogeneous manifolds. \S\ref{section:moves} develops a combinatorial language for describing fillings in Lie groups. \S\ref{section:multiamalgam} specializes to solvable Lie groups and reviews the theory of tame groups and Abels's multiamalgam. Finally, we prove our main theorem in \S\ref{section:main}.

\section{Acknowledgments.} We wish to thank Robert Young and Yves Cornulier for productive conversations. This work has been supported by NSF award 1148609.

\section{Preliminaries.}
\label{section:preliminaries}

\subsection{Filling.}
\label{subsection:filling}

In this section, $X$ will be a complete, simply connected Riemannian manifold admitting a transitive Lie group action by isometries. We will collect several facts about filling loops in $X$.  Everything in this section is both trivial and well known, but it seems easier to write the proofs than to find them in the literature. The key facts we will prove are as follows.

\begin{itemize}
\item Proposition \ref{proposition:smallscale} shows that $X$ is Lipschitz 1-connected on a small scale.  That is, there is some constant $D$ such that $\Span(\gamma)=O(\Lip(\gamma))$ when $\Lip(\gamma)<D$.
\item Corollary \ref{corollary:finitespan} shows that all loops in $X$ have Lipschitz fillings.
\item Proposition \ref{proposition:fillingspan} proves the existence a ``filling span function"---meaning that there is a function $C(X,R)$ such that $\Span(\gamma)<C(X,\Lip(\gamma))$ for all Lipschitz loops $\gamma:S^1\To X$.
\item Proposition \ref{proposition:QI} shows that Lipschitz 1-connectedness is a QI-invariant in this setting.
\end{itemize}

\subsection{Templates.}

Let $D^2$ be the unit disk, equipped with the Euclidean metric.  Purely as a matter of convenience, we will think of $D^2$ and the unit circle $S^1$ as subsets of $\mathbb{C}$. We will need some convenient cellular decompositions of $D^2$.

\begin{figure}[t]
\labellist
\small\hair 2pt

\endlabellist

\centering
\centerline{\psfig{file=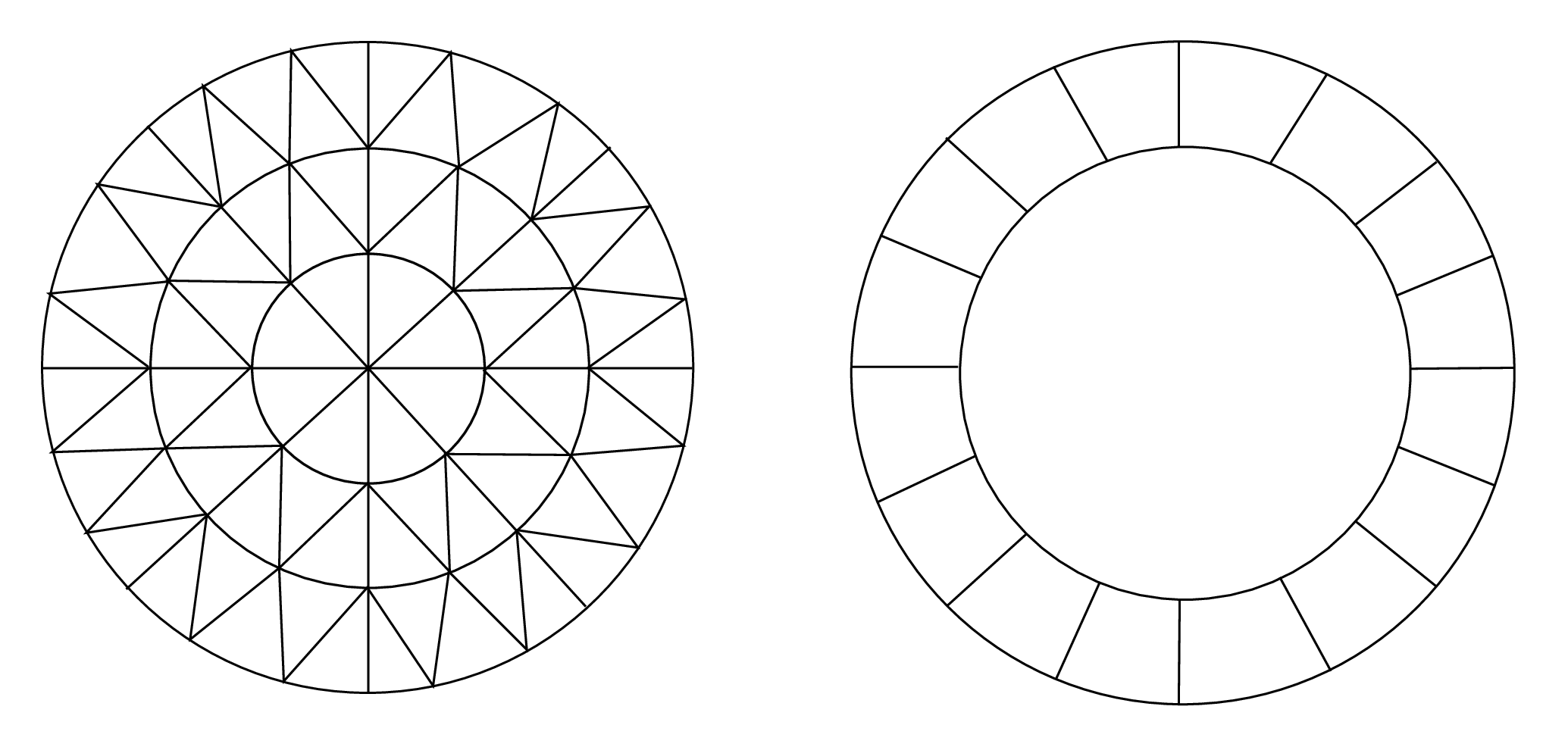,scale=60}}
\caption{On the right, a decomposition of the unit disk into $O(\epsilon^{-2})$ triangles uniformly bilipschitz to $\Deps$ as in Proposition \ref{proposition:web}. On the right, a decomposition of $D^2$ into a central disk together with $O(\epsilon^{-1})$ ``sectors" uniformly bilipschitz to $\Deps$ as in Proposition \ref{proposition:sun}}
\label{figure:templates}
\end{figure}

\begin{definition}
Let $\Deps$ be the Euclidean disk of radius $\epsilon$.
\end{definition}

\begin{proposition}
\label{proposition:web}
There exists a constant $C$ such that for any $0<\epsilon<1$ there is a triangulation of $D^2$ into at most $C\epsilon^{-2}$ triangles which are $C$-bilipschitz to $\Deps$.
\end{proposition}

See the left side of Figure \ref{figure:templates} for a decomposition of this type. The proof of Proposition \ref{proposition:web} is left to the reader.

\paragraph{Using templates.} Proposition \ref{proposition:web} will help us fill loops in the following way. Suppose that $\cC$ is a class of loops in $X$ such that $\Span(\gamma)=O(\Lip(\gamma))$ for $\gamma\in\cC$---typically, $\cC$ will consist of loops with small Lipschitz constant. Now, given some other loop $\gamma:S^1\To X$, we may wish to find a $O(\Lip(\gamma))$-Lipschitz filling $f:D^2\To X$ of $\gamma$. We can often use the following abstract strategy to construct $f$.
\begin{itemize}
\item For some sufficiently small $\epsilon$, take a decomposition of $D^2$ as in Proposition \ref{proposition:web}. For a 2-cell $\Delta$ of this decomposition, let $\psi_\Delta:\Delta\To\Deps$ be a $C$-bilipschitz map, where $C$ is the universal constant guaranteed by the proposition.
\item Take $f$ to be $\gamma$ on $\partial D^2$, and extend $f$ over the 1-skeleton of the decomposition in such a way that $\Lip(f)=O(\Lip(\gamma))$ and the restriction of $f$ to the boundary of each 2-cell represents an element of $\cC$. That is, we want $f$ to be such that for every 2-cell $\Delta$, the map $\gamma_\Delta:S^1\To X$ given by
$$\gamma_\Delta:e^{i\theta}\mapsto f|_{\partial\Delta}\circ\psi_\Delta^{-1}(\epsilon e^{i\theta})$$
satisfies $\gamma_\Delta\in\cC$.
\item For each 2-cell $\Delta$, observe that $\Lip(\gamma_\Delta)=O(\epsilon\Lip(\gamma))$, so there exists a $O(\epsilon\Lip(\gamma))$-Lipschitz filling $f_\Delta:D^2\To X$ of $\gamma_\Delta$.
\item Extend $f$ over $\Delta$ with Lipschitz constant $O(\Lip(\gamma))$ by taking, for each 2-cell $\Delta$,
$$f|_\Delta(z)=f_\Delta\left(\frac{1}{\epsilon}\psi_\Delta(z)\right).$$
\end{itemize}
The nontrivial part of this type of argument comes when we try to extend $f$ over the 1-skeleton with the desired properties, so we will typically suppress the other details. 

\begin{proposition}
\label{proposition:sun}
There exists a constant $C$ such that, for each $0<\epsilon<1$, the unit disk $D^2$ may be cellularly decomposed into an inner disk of radius $1-\epsilon$, surrounded by an annular region divided into 2-cells (we call these sectors) with the following properties.
\begin{itemize}
\item Each sector is bounded by two radial line segments, an arc of $\partial D^2$, and an arc of the boundary of the inner disk.
\item Sectors are $C$-bilipschitz to $\Deps$.
\item The number of sectors is between $\frac{1}{C\epsilon}$ and $\frac{C}{\epsilon}$.
\end{itemize}
\end{proposition}

See the right side of Figure \ref{figure:templates} for a decomposition of this type.  The proof of Proposition \ref{proposition:sun} is left to the reader. This proposition will typically used to convert a filling of a loop $\tilde{\gamma}$ to a filling $f$ of a nearby loop $\gamma$, by taking $f$ restricted to the inner disk to be a slightly rescaled filling of $\tilde{\gamma}$.

\subsection{Basic results on filling in homogeneous manifolds.}
\begin{proposition}
\label{proposition:smallscale}
Let X be a simply connected, complete homogeneous Riemannian manifold. There exist constants $C(X)$ and $D(X)$ such that the following hold.
\begin{itemize}
\item If $x,y\in X$ and $d(x,y)\leq D(X)$, then there is a unique geodesic in $X$ connecting $x$ to $y$.
\item Suppose $\gamma:S^1\To X$ is such that $0<\Lip(\gamma)<D(X)$.  Then $\Span(\gamma)<C(X)\Lip(\gamma)$.
\end{itemize}
\end{proposition}

\begin{proof}
It is clear that there is a uniform upper bound on sectional curvatures of $X$, since $X$ is homogeneous. Therefore, $X$ is a CAT($\kappa$) space for some $\kappa\geq 0$ \cite{bh}[Theorem 1A.6]. This implies \cite{bh}[Proposition 1.4(1)] that there is some constant $D(\kappa)$ depending only on $\kappa$ such that points of $X$ separated by less than $D(\kappa)$ are connected by a unique geodesic, so taking $D(X)<D(\kappa)$ ensures unique geodesics.

By the existence of normal coordinates in Riemannian manifolds \cite[Proposition 8.2]{kn}, there exists a bilipschitz map $\psi:U\To V$ where $U$ is a neighborhood in $X$ and $V$ a neighborhood in $\RR^{\dim(X)}$. Take $D>0$ small enough that there is some $x\in U$ with $U\supset B_D(x)$ and $\psi(B_D(X))\subset V_0\subset V$ for some convex set $V_0$. Now let $\gamma:S^1\To X$ be $D$-Lipschitz with $\gamma(1)=x$. Without loss of generality, $\psi(x)=0$, and we may fill $\psi\circ\gamma$ by coning off--that is, we define a $\Lip(\psi)\Lip(\gamma)$-Lipschitz filling $f_0:D^2\To V_0$ of $\psi\circ\gamma$ by letting $f_0(re^{i\theta})=r\psi(\gamma(e^{i\theta}))$. If we let $f=\psi^{-1}\circ f_0$, then $f$ is a filling of $\gamma$ and $\Lip(f)\leq \Lip(\psi^{-1})\Lip(\psi)\Lip(\gamma)$. Taking $D(X)<D$ and $C(X)>\Lip(\psi^{-1})\Lip(\psi)$, the result follows.
\end{proof}

\begin{corollary}
\label{corollary:finitespan}
For all Lipschitz maps $\gamma:S^1\To X$, we have $\Span(\gamma)<\infty$.
\end{corollary}

\begin{proof}
Let $\tilde{f}:D^2\To X$ be a (continuous) filling of a Lipschitz $\gamma:S^{1}\To X$. We must find a Lipschitz filling $f$ of $\gamma$. For $\epsilon>0$, let $\tau_\epsilon$ be the triangulation of $D^2$ given in Proposition \ref{proposition:web}, and for each 2-cell $\Delta$ of $\tau_\epsilon$, let $\psi_\Delta:\Delta\to\Deps$ be the $C$-bilipschitz map guaranteed by the proposition. Let $F(\epsilon)$ be the largest value of $d(\tilde{f}(x),\tilde{f}(y))$ such that $x,y$ are adjacent vertices of $\tau_\epsilon$.  If $F(\epsilon)\To 0$ as $\epsilon\To 0$, then Proposition \ref{proposition:smallscale} lets us produce a Lipschitz filling $f$ of $\gamma$ as follows.  Set $f(x)=\tilde{f}(x)$ for each vertex $x$ of $\tau_\epsilon$, where $\epsilon$ is small enough that $CF(\epsilon)<D(X)$ where $C$ is the constant given by Proposition \ref{proposition:web} and $D(X)$ is the constant given by Proposition \ref{proposition:smallscale}. Set $f$ to be a constant speed geodesic on each edge of $\tau_\epsilon$, so that the map $S^1\To X$ given by
$$e^{i\theta}\mapsto f\left(\psi_\Delta^{-1}\left(\epsilon e^{i\theta}\right)\right)$$
is $D(X)$-Lipschitz. By Proposition \ref{proposition:smallscale}, $f$ now admits a Lipschitz extension over 2-cells.

On the other hand, if $F(\epsilon)$ does not go to $0$ as $\epsilon\To 0$, then for every natural number $n$, we may find points $x_n,y_n\in D^2$ such that $d(x_n,y_n)<\frac{1}{n}$ but $d(\tilde{f}(x_n),\tilde{f}(y_n))>\delta$ for some fixed $\delta>0$. Because $D^2\times D^2$ is compact, we know that $(x_n,y_n)$ must subconverge to some point $(x,x)$ in the diagonal of $D^2\times D^2$. But $(\tilde{f}(x_n),\tilde{f}(y_n))$ cannot subconverge to $(\tilde{f}(x),\tilde{f}(x))$ because $d(\tilde{f}(x_n),\tilde{f}(y_n))>\delta$. This contradicts continuity of $\tilde{f}$. 
\end{proof}

\begin{proposition}
\label{proposition:fillingspan}
Let X be a simply connected, complete homogeneous Riemannian manifold. For any $L>0$, there exists $C(X,L)>0$ such that if $\gamma:S^1\To X$ is $L$-Lipschitz, then $\Span(\gamma)\leq C(X,L)$.
\end{proposition}

Gromov \cite[\S 5]{gromov} refers to $C(X,L)$ as the filling span function of $X$.

\begin{proof}
Fix $L$ (all constants from here on will be presumed to depend uncontrollably on $L$). Let $\cC_L$ denote the $L$-Lipschitz loops in $X$ with some fixed basepoint $x$, and equip $\cC_L$ with the uniform metric. By Arzela-Ascoli, $\cC_L$ is compact. Certainly, $\Span$ is not continuous on $\cC_L$, but we will show that it is bounded. Because $\cC_L$ is compact, for every $\epsilon>0$, $\cC_L$ may be covered by a finite number of $\epsilon$ balls.  Hence, it suffices to find a constant $C>1$ such that for $\gamma_0,\gamma\in\cC_L$ with $d(\gamma,\gamma_0)<\frac{1}{C}$ we have
$$\Span(\gamma)<C\max\{1,\Span(\gamma_0)\}.$$
To do this, let $D(X)$ be as in proposition \ref{proposition:smallscale} (assuming without loss of generality that $D(X)<1$), and let $\gamma_0,\gamma\in\cC_L$ with $d(\gamma_0,\gamma)<D(X)$. Let $f_0:D^2\To X$ be a Lipschitz filling of $\gamma_0$. Let $\epsilon=\frac{D(X)}{L}$ and decompose $D^2$ into an inner disk and sectors as in Proposition \ref{proposition:sun}. We will produce a filling $f:D^2\To X$ of $\gamma$ which is $\frac{\Lip(f_0)}{1-\epsilon}$-Lipschitz on the inner disk and $O(1)$-Lipschitz on the annular region as desired.

Define $f|_{S^1}$ to be equal to $\gamma$. Let $R=1-\epsilon$ be the radius of the inner disk and define $f$ to be a rescaled copy of $f_0$ on the inner disk---i.e., $f(re^{i\theta})=f_0(\frac{r}{R}e^{i\theta})$ for $r\leq R$. This implies that $f$ is $\frac{\Lip(f_0)}{R}$-Lipschitz on the inner disk. If $x$ and $y$ are the images under $f$ of the endpoints of a radial segment separating two sectors, then $d(x,y)<D(X)$ because $d(\gamma,\gamma_0)<D(X)$. Define $f$ to be a minimal speed geodesic on each of the radial segments separating two sectors, so that $f$ is $O(1)$-Lipschitz on the boundary of each sector (considering $L$ as fixed). Since sectors are uniformly bilipschitz to $\Deps$, $f$ admits a $O(1)$-Lipschitz extension over each sector by Proposition \ref{proposition:smallscale}.

We see that $f$ is $O(1)$-Lipschitz on the annular region and $\frac{\Lip(f_0)}{1-\epsilon}$-Lipschitz on the inner disk, implying the desired bound for $\Span(\gamma)$, taking any sufficiently large $C$.
\end{proof}

\begin{proposition}
\label{proposition:QI}
Let $X$ and $Y$ be simply connected, complete, homogeneous Riemannian manifolds, and suppose that $X$ is quasi-isometric to $Y$. If $Y$ is Lipschitz 1-connected, then so is $X$.
\end{proposition}

It seems likely that $X$ and $Y$ are bilipschitz if they are quasi-isometric.

\begin{proof}
Let $\psi:X\To Y$ be a quasi isometry, and $\Psi:Y\To X$ a quasi inverse to $\psi$.
Let $\gamma:S^1\To X$ be an $L$-Lipschitz loop, where $L>1$ without loss of generality, and let $\epsilon=\frac{1}{L}$. To obtain a $O(L)$-Lipschitz filling for $\gamma$, we proceed as follows.

Using Proposition \ref{proposition:web}, subdivide $D^2$ into $O(L^2)$ triangles bilipschitz to $\Deps$, so that adjacent vertices on the boundary are mapped to within $O(1)$ of each other by $\psi\circ\gamma$. Let $\tilde{\gamma}:S^1\To Y$ be a $O(L)$-Lipschitz loop in $Y$ which agrees with $\psi\circ\gamma$ on vertices of our triangulation. By assumption, $Y$ is Lipschitz 1-connected, so $\tilde{\gamma}$ admits a $O(L)$-Lipschitz filling $\tilde{f}:D^2\To Y$.

We wish to convert $\Psi\circ\tilde{f}$ into a filling $f:D_2\To X$ of $\gamma$.  Let $f_0:D^2\To X$ be a $O(L)$-Lipschitz map which agrees with $\Psi\circ\tilde{f}$ on vertices of our triangulation---such a map exists because we may fill edges with constant speed geodesics and then fill triangles in $X$ by Proposition \ref{proposition:fillingspan}, as $\Psi\circ\tilde{f}$ maps adjacent vertices to within $O(1)$ of each other. This gives us a $O(L)$-Lipschitz filling of $f_0|_{S^1}$. Note that the distance from $\gamma$ to $f_0|_{S^1}$ in the uniform metric is $O(1)$ because on $f_{0}|_{S^1}$ agrees with $\Psi\circ\psi\circ\gamma$ on vertices.

Now we take a new subdivision of $D^2$, using Proposition \ref{proposition:sun} to subdivide $D^2$ into an inner disk surrounded by $O(L)$ sectors which are uniformly bilipschitz to $\Deps$. We build our filling $f$ of $\gamma$ as follows. On the inner disk, let $f$ be given by a rescaled copy $f_0$. On radial segments, take $f$ to be a constant speed geodesic, and fill sectors using Proposition \ref{proposition:fillingspan}.
\end{proof}

\section{Lipschitz moves.}
\label{section:moves}

We now specialize to the case where $X$ is equal to some simply connected Lie group $G$ equipped with a left-invariant Riemannian metric. Our main goal in this section is to reduce questions about filling loops in $G$ to questions about manipulating words in some compact generating set for $G$.

\paragraph{Notation.} Any simply connected Lie group $G$ admits a compact generating set $\cS$. The set $\cS^*$ consists of all words $s_1s_2\ldots s_\ell$ where $s_1,\ldots,s_\ell\in\cS$ and $\ell\in\NN$, together with the empty word $\varepsilon$. The length of a word $w\in\cS^*$ will be denoted $\ell(w)$. Given $w=s_0\ldots s_\ell\in\cS^*$, $w^{-1}$ denotes the word $s_\ell^{-1}\ldots s_1^{-1}$. If $w\in \cS^*$ represents the identity element $1_G$ of $G$, $w$ is said to be a relation. If $w,w'\in\cS^*$ represent the same element of $G$, we write $w=_G w'$. The word norm with respect to $\cS$ will be denoted by $|\cdot|_\cS$. That is for $g\in G$, we define $|g|_\cS$ to be $\inf\{\ell(w):g=_G w\in\cS^*\}$.

\paragraph{Assumptions.} Any time we take a compact generating set $\cS$ for $G$, we assume without loss of generality that $\cS$ is symmetric---meaning $s\in\cS\Leftrightarrow s^{-1}\in\cS$---and that $1_G\in\cS$, unless otherwise indicated.

\paragraph{It suffices to fill over the unit square.} It will be convenient to consider loops as maps $[0,1]\To X$ rather than $S^1\To X$ and fillings as maps from the unit square $[0,1]\times[0,1]\To X$ rather than $D^2\To X$.

\begin{definition}
Given $\beta:[0,1]\To X$ with $\beta(0)=\beta(1)$, a filling of $\beta$ over the unit square is a map $f:[0,1]\times[0,1]\To X$ with the following properties:
\begin{itemize}
\item $f$ agrees with $\beta$ on the bottom edge of the unit square, meaning that $f(t,0)=\beta(t)$ for all $t\in [0,1]$.
\item $f$ is constant on the set on the other three edges of the unit square, meaning that $f(x,y)=\beta(0)$ for all
$$(x,y)\in ([0,1]\times {1})\cup ({0}\times[0,1])\cup({1}\times [0,1]).$$
\end{itemize}
We define $\Span(\beta)$ to be the infimum of $\Lip(f)$ as $f$ ranges over Lipschitz fillings of $\beta$ over the unit square.

Similarly, given $\beta,\gamma:[0,1]\To X$, a homotopy from $\beta$ to $\gamma$ is a map $[0,1]\times[0,1]\To X$ such that $f(t,1)=\gamma(t)$, $f(t,0)=\beta(t)$ and the restrictions $f|_{0\times[0,1]}$ and $f|_{1\times[0,1]}$ are constant.
\end{definition}

We will now see that it makes no difference whether we consider loops as maps from $[0,1]$ or $S^1$. The key tool is the following lemma.

\begin{lemma}
\label{lemma:reparameterization}
There exists a constant $C$ such that, given loops $\beta:S^1\To X$ and $\gamma:S^1\To X$ such that $\gamma$ is a reparameterization of $\beta$, and given a filling $f:D^2\To X$ of $\gamma$, we have
$$\Span(\beta)\leq C\max\{\Lip(f),\Lip(\beta)\}.$$
\end{lemma}

\begin{proof}
An equivalent statement is proved in \cite[Lemma 8.13]{ry}. 
\end{proof}

\begin{corollary}
\label{corollary:square}
There is a universal constant $C>0$ with the following property. Suppose that $\gamma:S^1\To X$ is a Lipschitz loop and $\beta:[0,1]\To X$ a Lipschitz path with $\beta(t)=\gamma(e^{2\pi it})$ for all $t\in [0,1]$. Then
$$\frac{\Span(\beta)}{C}<\Span(\gamma)<C\Span(\beta).$$
\end{corollary}

\begin{proof}
Let $\beta_0:\partial([0,1]^2)\To X$ be equal to $\beta$ on the bottom edge and $\beta(0)$ on the other three edges, and fix a bilipschitz map $\psi:D^2\To [0,1]\times[0,1]$.  Observe that $\gamma$ is a reparameterization of $\beta_0\circ\psi:S^1\To X$. If $f:D^2\To X$ is a Lipschitz filling of $\gamma$, then by Lemma \ref{lemma:reparameterization} there is some $O(\Lip(f))$-Lipschitz filling $\tilde{f}:D^2\To X$ of $\beta_0\circ\psi$, so $\tilde{f}\circ\psi^{-1}$ is a $O(\Lip(f))$-Lipschitz filling of $\beta$. We see that $\Span(\beta)=O(\Span(\gamma))$. The reverse inequality is proved similarly.
\end{proof}

\subsection{Filling relations.}

\paragraph{It suffices to fill words.} Let $\cS$ be a compact generating set for $G$. For each $s\in \cS$ choose a Lipschitz curve $\gamma_s:[0,1]\To G$ connecting $1$ to $s$, such that the following properties hold.
\begin{itemize}
\item There is some uniform bound on $\Lip(\gamma_{s})$ as $s$ ranges over $\cS$.
\item $\gamma_{1_G}$ is constant.
\item $\gamma_s(1-t)=\gamma_{s^{-1}}(t)$ for all $s\in \cS$ and $t\in[0,1]$.
\end{itemize}
Given a word $w=s_1s_2\ldots s_\ell\in\cS^*$, let $\gamma_{w}:[0,1]\To G$ denote the concatenation of the paths $\gamma_{s_1},s_1\gamma_{s_2},\ldots,s_1\ldots s_{\ell-1}\gamma_{s_\ell}$, reparameterized so that the $i$-th of these paths is used on $[\frac{i-1}{\ell},\frac{i}{\ell}]$. That is, for $0\leq t \leq 1$ and $i=1,\ldots \ell$, we have
$$\gamma_w\left(\frac{i-1+t}{\ell}\right)=s_1s_2\ldots s_{i-1}\gamma_{s_i}(t).$$

\begin{proposition}
\label{proposition:words}
Suppose that there exists a constant $C$ such that for every $w\in\cS^*$ with $w=_G 1_G$, there exists a $C\ell(w)$-Lipschitz filling of $\gamma_{w}$ over the unit square. Then $G$ is Lipschitz 1-connected.
\end{proposition}

\begin{proof}
Note that $G=\bigcup_{k=1}^{\infty}\cS^k$, and each $\cS^k$ is compact. Hence, by the Baire category theorem, some power $\cS^k$ must contain an open neighborhood, so $\cS^{2k}$ contains some open neighborhood of the identity, which in turn contains the $r$-ball around the identity for some $r>0$. Take $K$ to be a natural number larger than $\frac{2k}{r}$, so that $\cS^K$ contains the $1$-ball around the identity in $G$.

Given a loop $\gamma:[0,1]\To G$, we produce a $O(\Lip(\gamma))$-Lipschitz filling $f:[0,1]^2\To G$ of $\gamma$ as follows. Let $n=\lceil\Lip(\gamma)\rceil$---by Propositions \ref{proposition:smallscale} and \ref{proposition:fillingspan} it suffices to only consider $\gamma$ such that $n$ is at least $2$. Cellularly decompose the unit square into the rectangle $[0,1]\times\left[\frac{1}{n},1\right]$ together with the collection of squares $\left[\frac{i}{n},\frac{i+1}{n}\right]\times\left[0,\frac{1}{n}\right]$ as $i$ runs from $0$ to $n-1$.

Define $f$ to be constant on all the vertical edges. Note that for each $i$, the element $\gamma\left(\frac{i}{n}\right)^{-1}\gamma\left(\frac{i+1}{n}\right)$ lies in the $1$-ball in $G$, so it may be represented by some $w_i\in\cS^K$. On the horizontal edge connecting $\left(\frac{i}{n},\frac{1}{n}\right)$ to $\left(\frac{i+1}{n},\frac{1}{n}\right)$, let $f$ be a copy of $\gamma_{w_i}$ translated by $\gamma\left(\frac{i}{n}\right)$---that is for $0\leq t \leq 1$, take $f\left(\frac{i+t}{n},\frac{1}{n}\right)=\gamma\left(\frac{i}{n}\right)\gamma_{w_i}(t)$. Letting $w=w_1\ldots w_n$, observe that $f$ agrees with $\gamma_w$ the the bottom of the rectangle $[0,1]\times\left[\frac{1}{n},1\right]$, so $f$ may be extended over this rectangle with
$$\Lip(f)=O(\Lip(\gamma_w))=O(\ell(w))=O(\Lip(\gamma))$$
where the first bound comes from our hypothesis.  Finally, $f$ admits a $O(\Lip(\gamma))$-Lipschitz extension over each square by Proposition \ref{proposition:fillingspan}.
\end{proof}

\begin{definition}
Suppose $\cC$ is a collection of pairs $(v,w)$ where $v,w\in\cS^*$ are words in $\cS$. We say that $v\leadsto w$ for $(v,w)\in\cC$ if there exists some constant $C$ depending only on $\cC$ with the following property: given $(v,w)\in\cC$, the map from $\partial([0,1]\times [0,1])$ to $X$ which is equal to $\gamma_v$ on the bottom edge, $\gamma_w$ on the top edge, and constant on the sides admits a $C\ell(v)$-Lipschitz filling.
\end{definition}

We now give some basic rules for manipulating fillings. Often, the set $\cC$ will be inferred from context. 

\begin{lemma}
\label{lemma:big}
The following rules hold in any simply connected Lie group $G$.
\begin{itemize}
\item Suppose $\cC\subset\cS^*\times\cS^*$. Then $v\leadsto w$ for $(v,w)\in\cC$ if and only if $vw \leadsto \varepsilon$ for $(v,w)\in\cC$ and $\Lip(w)=O(\Lip(v))$ for each pair $(v,w)\in\cC$.
\item Fix $n$, and suppose given sets $\cC_i\subset\cS^*\times\cS^*$ for $i=1$ to $n$. If, for each $i=1$ to $n$, we have $v\leadsto w$ for $(v,w)\in\cC_i$, then $v_1\ldots v_n \leadsto w_1\ldots w_n$ for all $(v_1,w_1)\in\cC_1,\ldots,(v_n,w_n)\in\cC_n$.
\item Given $\cC,\cC'\subset\cS^*\times\cS^*$, if $v_1\leadsto v_2$ for $(v_1,v_2)\in\cC$ and $v_2\leadsto v_3$ for $(v_2,v_3)\in\cC'$, then $v_1\leadsto v_3$ for pairs $(v_1,v_3)$ such that there exists $v_2\in\cS^*$ with $(v_1,v_2)\in\cC$ and $(v_2,v_3)\in\cC'$.
\item $ww^{-1}\leadsto\varepsilon$ for $w\in\cS^*$.
\item Let $U\subset G$ be an bounded neighborhood of the identity in $G$. If $\mathcal{D}$ is the collection of relations $w=w_1\ldots w_\ell\in\cS^*$ such every prefix $w_1\ldots w_i$ represents an element of $U$, then $w\leadsto \varepsilon$ for $w\in\mathcal{D}$.
\item $G$ is Lipschitz 1-connected if and only if and $v\leadsto\varepsilon$ for $v\in\cS^*$. 
\end{itemize}
\end{lemma}

\begin{proof}
The first item follows immediately from Lemma \ref{lemma:reparameterization}.

To prove the second item, subdivide the unit square into rectangles of the form $\left[\frac{i-1}{n},\frac{i}{n}\right]\times[0,1]$ for $i=1,\ldots,n$. Suppose given $(v_1,w_1)\in\cC_1,\ldots,(v_n,w_n)\in\cC_n$. By hypothesis, there exists a homotopy $f_i:[0,1]\times [0,1]\To G$ from $v_i$ to $w_i$ with $\Lip(f_i)=O(\Lip(v_i))$. Define a map $f:[0,1]\times[0,1]$ by putting a rescaled copy of $f_i$ in $\left[\frac{i-1}{n},\frac{i}{n}\right]\times[0,1]$---that is, for $0\leq t \leq 1$ and $i=1,\ldots,n$, we take
$$f\left(\frac{i-1+t}{n},y\right)=f_i(t,y).$$
Thus $\Lip(f)=O(\max\{n\Lip(f_i):i=1\ldots n\})=O(\Lip(v))$ because $n$ is fixed. Since $f$ restricted to the top of the unit square is a reparameterization of $v_1\ldots v_n$, and $f$ restricted to the bottom of the unit square is a reparameterization of $w_1\ldots w_n$, we are done.

The proof of the third item is similar: divide the unit square into two rectangles $[0,1]\times\left[0,\frac{1}{2}\right]$ and $[0,1]\times\left[\frac{1}{2},1\right]$ and put a rescaled copy of the homotopy from $v_1$ to $v_2$ into the bottom rectangle and a rescaled copy of the homotopy from $v_2$ to $v_3$ into the top rectangle to get the desired homotopy $v_1$ to $v_3$.

To prove the fourth item, let $w\in\cS^*$, and let $f:[0,1]\times[0,1]$ be defined by
$$f(t,s)=\gamma_w(2\max\{0,\min\{t-s,1-t-s\}\}).$$
The reader may check that $f$ is a $\Lip(\gamma_{ww^{-1}})$-Lipschitz filling of $\gamma_{ww^{-1}}$.

To prove the fifth item, let $K$ be such that $\cS^K$ contains $U$. Given $w\in\cD$, let $\ell=\ell(w)$ so that we can write $w=s_1\ldots s_\ell$. We will find a $O(\ell)$ filling $f$ of $\gamma_w$ as follows. Subdivide $[0,1]\times[0,1]$ into the rectangle $[0,1]\times \left[\frac{1}{\ell},1\right]$ together with the squares of the form $\left[\frac{i-1}{\ell},\frac{i}{\ell}\right]\times\left[0,\frac{1}{\ell}\right]$. Define $f$ to be constant on $[0,1]\times \left[\frac{1}{\ell},1\right]$. For each $i$, choose $w_i\in\cS^K$ such that $w_i=_G s_1\ldots s_i$ and take $f$ to be $\gamma_{w_i}$ on the vertical edge $\frac{i}{\ell}\times\left[0,\frac{1}{\ell}\right]$---that is
$$f\left(\frac{i}{\ell},\frac{1-t}{\ell}\right)=\gamma_{w_i}(t)$$
for $0\leq t \leq 1$ and $i=0,\ldots,n$. We may extend $f$ over each square with Lipschitz constant $O(\ell)$ by Proposition \ref{proposition:fillingspan}.

The sixth item is a consequence of Proposition \ref{proposition:words}.
\end{proof}

\subsection{Normal form triangles.}

We now discuss normal forms $\omega$ and $\omega$-triangles. Using a technique of Gromov, we shall see that $G$ is Lipschitz 1-connected if $\omega$-triangles are Lipschitz 1-connected.

\begin{definition}
A normal form for a compactly generated group $G$ equipped with compact generating set $\cS$ is a map
$$\omega:G\To \cS^*$$
such that $\ell(\omega(g))=O(|g|_\cS)$ for $g\in G$. If $\omega$ is a normal form, then an $\omega$-triangle is a word in $\cS^*$ of the form
$$\omega(g_1)\omega(g_2)\omega(g_3)$$
where $g_1 g_2 g_3=_{G}1.$
\end{definition}

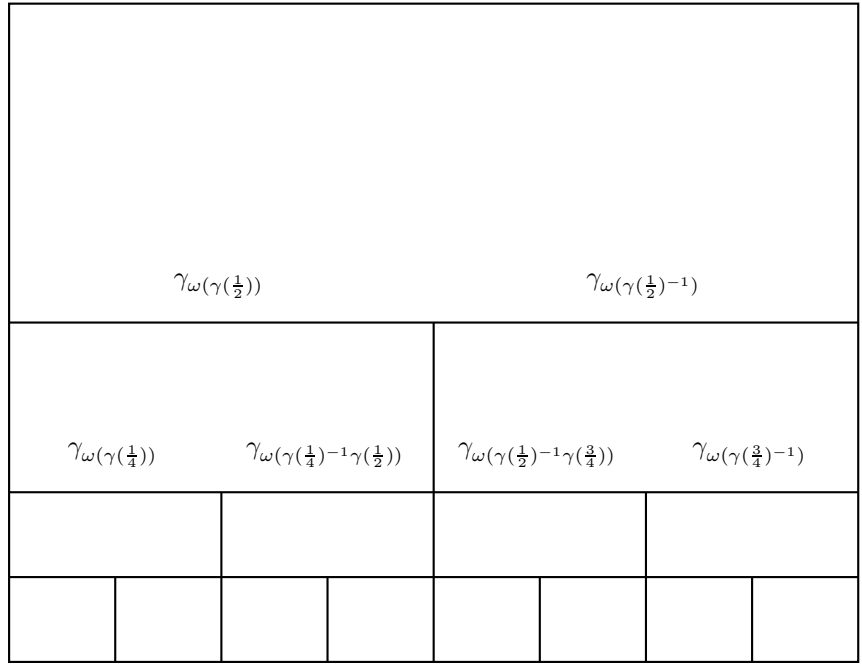
\begin{figure}[t]
\labellist
\small\hair 2pt


\pinlabel $\gamma_{\omega(\gamma(\frac{1}{2}))}$ at 215 162
\pinlabel $\gamma_{\omega(\gamma(\frac{1}{2})^{-1})}$ at 375 162
\pinlabel $\gamma_{\omega(\gamma(\frac{1}{4}))}$ at 175 98
\pinlabel $\gamma_{\omega(\gamma(\frac{1}{4})^{-1}\gamma(\frac{1}{2}))}$ at 255 98
\pinlabel $\gamma_{\omega(\gamma(\frac{1}{2})^{-1}\gamma(\frac{3}{4}))}$ at 335 98
\pinlabel $\gamma_{\omega(\gamma(\frac{3}{4})^{-1})}$ at 415 98
\endlabellist

\centering
\centerline{\psfig{file=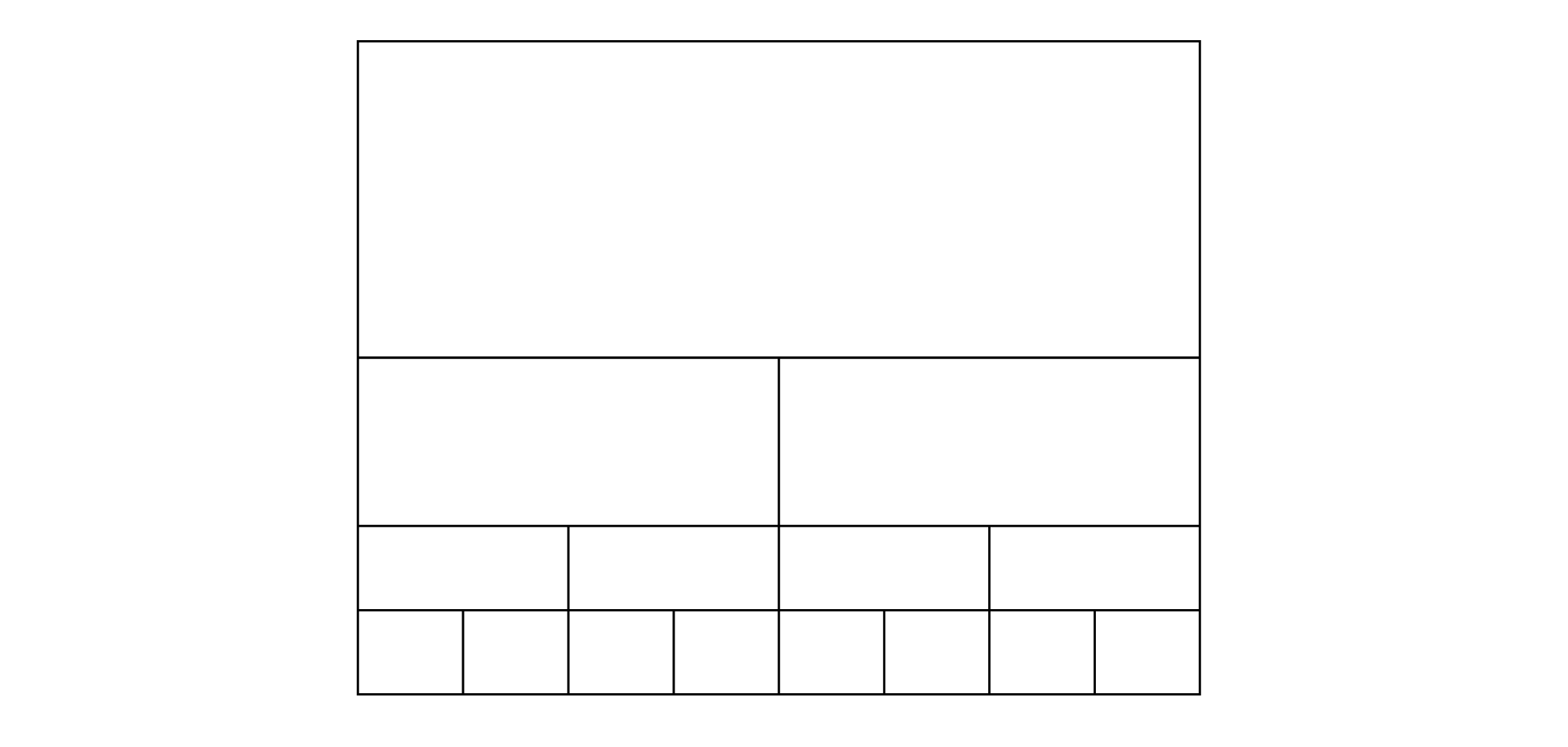,scale=100}}
\caption{This figure indicates how to fill $\gamma=\gamma_w$ for an arbitrary relation $w\in\cS^*$ given that one knows how to fill $\omega$-triangles. The top edge and all the vertical edges are taken to be constant, and each horizontal edge is understood to be an appropriate translate of its label, so that each rectangle represents an $\omega$-triangle, except the bottom row. The bottom edge is taken to be $\gamma_w$, and there is a row of squares along the bottom which may be filled by Proposition \ref{proposition:fillingspan}}
\label{figure:gromov}
\end{figure}

\begin{lemma}
\label{lemma:gromov}
Let $\omega:G\To\cS^*$ be a normal form. If $\Delta\leadsto \varepsilon$ for $\omega$-triangles $\Delta$, then $G$ is Lipschitz 1-connected.
\end{lemma}

\begin{proof}
This is proved in \cite[Proposition 8.14]{ry}. The proof is sketched in Figure \ref{figure:gromov}.
\end{proof}

\section{Tame subgroups and the multiamalgam.}
\label{section:multiamalgam}

\paragraph{Assumptions.} From here on, we specialize to the case where $G=U\rtimes A$, where $U$ and $A$ are contractible Lie groups, with $A$ abelian and $U$ a closed group of strictly upper triangular real matrices.

\paragraph{Standard solvable groups.} Observe that $A$ acts on the abelianization $U/[U,U]$. Fix a norm on the vector space $U/[U,U]$. If there is no vector $X$ such that $\lim_{n\To\infty}\frac{1}{n}\log
\|a^n\cdot X\|\To 0$ for all $a\in A$, we say that $G$ is standard solvable.

In this section we will be interested in the structure and geometry of standard solvable groups. \S \ref{subsection:weights} will describe the so-called standard tame subgroups of a standard solvable group. Lemma \ref{lemma:freehomotopy} will show that how to find Lipschitz fillings for words which already represent the identity in the free product of the standard tame subgroups. Theorem \ref{theorem:multiamalgam}---quoted from \cite{ct}---will give conditions under which $G$ may be presented as the free product of its standard tame subgroups modulo certain easily understood amalgamation relations.

\subsection{Weights.}
\label{subsection:weights}
Let $G=U\rtimes A$ be standard solvable, and let $\mfu$ be the Lie algebra of $U$, identified as usual with the tangent space of $U$ at $1_U$, and fix any norm $\|\cdot\|$ on $\mfu$. For $a\in A$ we denote the conjugation action of $a$ on $\mfu$ by $\Ad(a)$, so that $\Ad(a)X=\frac{d}{dt}\big|_{t=0}a^{-1}\exp(tX)a$. Observe that $\Hom(A,\RR)$ is a vector space.

\begin{definition}
(See \cite[\S 4.B]{ct}). For a homomorphism $\alpha:A\To \RR$, define the $\alpha$-weight space $\mfua\subset\mfu$ to consist of $0$ together with all $X\in\mfu$ such that for all $a\in A$,
$$\lim_{n\To\infty}\frac{1}{n}\log\|\Ad(a)^nX\|=\alpha(a).$$
Define the set of weights $\cW$ to consist of all $\alpha\in\Hom(A,\RR)$ for which $\dim\mfua>0$.
By a conic subset, we mean the intersection of $\cW$ with an open, convex cone in $\Hom(A,\RR)$ that does not contain $0$. Denote the set of all conic subsets by $\cC$. For $C\in \cC$, let $U_C$ denote the closed connected subgroup of $U$ whose Lie algebra is $\bigoplus_{\alpha\in C}\mfua$, and let $G_C=U_C\rtimes A$. These groups $G_C$ are referred to as standard tame subgroups of $G$ (see remarks in \S \ref{subsection:tame}).
\end{definition}

As an exercise, the reader may wish to compute the weights and weight spaces for a group of Sol type. We have that $\cC$ is finite, that $\mfu=\bigoplus_{\alpha\in\cW}\mfua$, and that $[\mfua,\mfub]\subset \mfu_{\alpha+\beta}$ for $\alpha,\beta\in\cW$ \cite[\S 4.B]{ct}. We now define $H_2(\mfu)$ and recall the definition of $\Kill(\mfu)$ so that we will be able to state Theorem \ref{theorem:MAIN}, our main theorem.

\begin{definition}
\label{definition:killing}
Let $d_3:\bigwedge^3\mfu\To\bigwedge^2\mfu$ and $d_2:\bigwedge^2\mfu\To\mfu$ be the maps of $A$-modules induced by taking
$$d_3(x\wedge y \wedge z)=[x,y]\wedge z+[y,z]\wedge x+[z,x]\wedge y$$
$$d_2(x\wedge y)=-[x,y].$$
Define $H_2(\mfu)=\ker(d_2)/\image(d_3).$

Define $\Kill(\mfu)$ to be the quotient of the symmetric square $\mfu\odot\mfu$ by the subspace spanned by elements of the form $[x,y]\odot z-y\odot[x,z]$.
\end{definition}

\paragraph{$H_2(\mfu)$ and $\Kill(\mfu)$ are $A$-representations.} Observe that the natural $A$-action on $\bigwedge^3\mfu$ descends to an $A$-action on $H_2(\mfu)$ because, by the Jacobi-like identity $\Ad(a)[X,Y]=[\Ad(a)X,Y]+[X,\Ad(a)Y]$, the subspaces $\image(d_3)$ and $\ker(d_2)$ are preserved by the action of $A$. Similarly, $\Kill(\mfu)$ is also an $A$-representation. Recall that, for an $A$-representation $V$ we define $V_0$ to consist of $0$ together with vectors $X$ such that $\lim_{n\To\infty}\frac{1}{n}\log\|a^{n}\cdot X\|=0$ for all $a\in A$. We thus define the subspaces $H_2(\mfu)_0$ and $\Kill(\mfu)_0$.

\subsection{Tame subgroups.}
\label{subsection:tame}

\begin{definition}
Given $a\in A$, a vacuum subset for $a$ is a compact $\Omega\subset U$ such that for every compact $K\subset U$, there is some $n>0$ with $\Ad(a)^nK\subset \Omega$.
We say that $G$ is tame if there exists $a\in A$ with a vacuum subset.
\end{definition}

$G$ is tame if and only if there is some $a\in A$ with $\alpha(a)<0$ for all $\alpha\in\cW$, so $G_C$ is tame for $C\in \cC$ \cite[Proposition 4.B.5]{ct}. We now wish to show that if $G$ is tame, then it is Lipschitz 1-connected.  Our starting point is the following.

\begin{proposition}
\label{proposition:square}
Suppose $G$ is tame, then there is some $a\in A$ and a compact generating set $\cS\subset U$ for $U$ such that $\Ad(a)\cS^2\subset \cS$.
\end{proposition}

\begin{proof}
By hypothesis, there is some $b\in A$ with a vacuum subset $\Omega$. Let $\cS_0$ be a compact generating set for $U$. As in the proof of Proposition \ref{proposition:words}, we see that some power of $\cS_0$ contains an open ball around the identity, hence for some $M>0$, the set $\cS_0^M$ contains $\Omega$.  As $\Omega$ is a vacuum set for $b$, there exists $L$ such that $\Ad(b)^L\cS_0^{2M}\subset\Omega$. Taking $a=b^L$ and $\cS=\cS_0^M$, we have that $\cS$ is a generating set because it contains $\cS_0$ (because $1\in\cS_0$ by our standing assumption that generating sets contain the identity), and $\Ad(a)\cS^2=\Ad(b)^L\cS_0^{2M}\subset\Omega\subset\cS$ as desired.
\end{proof}

\begin{proposition}
\label{proposition:tamelip1c}
If $G=U\rtimes A$ is tame, then $G$ is Lipschitz 1-connected.
\end{proposition}

\paragraph{Remark.} Probably, a tame group $G$ is CAT(0) for some choice of metric, but we do not know how to prove this, so we give a combinatorial proof using Lemma \ref{lemma:gromov}.

\begin{proof}
Fix $a\in A$ and a compact generating set $\cS_U\subset U$ such that $\Ad(a)\cS_U^2\subset\cS_U$ as in Proposition \ref{proposition:square}, let $\cS_A$ be a generating set for $A$ with $a\in\cS_A$, and let $\cS=\cS_U\cup\cS_A$, so that $\cS$ is a generating set for $G$. Note that $Ad(a)(s)=s$ for $s\in \cS_A$, and observe that if $\ell>\lceil\log_2 j\rceil>0$, then
$$\Ad(a)^\ell\left(\cS_U^j\right)\subset\Ad(a)^{\ell-1}\left(\cS_U^{\lceil\frac{j}{2}\rceil}\right)
\subset\Ad(a)^{\ell-2}\left(\cS_U^{\lceil\frac{1}{2}\lceil\frac{j}{2}\rceil\rceil}\right)
\subset$$
$$\ldots
\subset\Ad(a)^{\ell-\lceil\log_2 j\rceil}(\cS_U)\subset\cS_U,$$
because the function $f:j\mapsto\lceil\frac{j}{2}\rceil$ satisfies $f^{\lceil\log_2 j\rceil} j=1$ for natural numbers $j$. In other words, given $u\in U$, there exist $k\leq \lceil\log_2|u|_{\cS_U}\rceil$ and $s\in\cS_U$ such that $u=_G a^{k}sa^{-k}$. It follows, letting $\phi_A:G\To A$ denote projection, that we may define a normal form $\omega:G\To\cS^*$ such that
\begin{itemize}\item for any $g\in G$, the word $\omega(g)$ is given by $\omega(\phi_A(g))\omega(\phi_A(g)^{-1}g)$,
\item for $g\in A$, the word $\omega(g)\in\cS_A^*$ is a minimal length word representing $g$,
\item and for $g\in U\setminus\{1\}$, $\omega(g)$ is of the form $a^{k}sa^{-k}$ where $s\in\cS_U$ and $0\leq k=O(\log|g|_{\cS_U})$.\end{itemize}
To check that this is a normal form---i.e., that $\ell(\omega(g))=O(|g|_\cS)$, note that $|\phi_A(g)|_\cS\leq |g|_\cS$, and $|\phi_A(g)^{-1}g|_{\cS_U}=O(\exp|g|_\cS)$, so that
$$|\phi_A(g)^{-1}(g)|_\cS=O(\log(O(\exp|g|_\cS)))=O(|g|_\cS)$$
because $U$ is at most exponentially distorted in $G$ \cite[Proposition 6.B.2]{ct}.

It will suffice to show that $\Delta\leadsto \varepsilon$ for $\omega$-triangles $\Delta\in\cS^*$.  An $\omega$-triangle $\Delta$ has the form
$$a_1\omega(u_1)a_2\omega(u_2)a_3\omega(u_3),$$
where for $i=1,2,3$, we have $u_i\in U$ and $a_i\in\cS_A^*$ with the $a_i$ and $u_i$ satisfying $(\Ad(a_3a_2)u_1)(\Ad(a_3)u_2)u_3=1$. To show that $\Delta\leadsto\varepsilon$, it thus suffices to establish the following two facts.
\begin{itemize}
\item $\omega(u)b\leadsto b\omega(b^{-1}ub)$ for $u\in U$ and $b\in \cS_A^*$.
\item $\omega(u_1)\omega(u_2)\omega(u_3)\leadsto\varepsilon$ for $u_1,u_2,u_3\in U$ such that $u_1u_2u_3=1$.
\end{itemize} 
Lemma \ref{lemma:big} will be crucial for showing the first fact---in particular, we use the Lemma to provide homotopies between words which stay inside a bounded neighborhood of the identity.

\paragraph{Conjugation.} To show that $\omega(u)b\leadsto b\omega(b^{-1}ub)$ for $u\in U$ and $b\in \cS_A^*$, note that $\omega(u)$ may be written as $a^{k}sa^{-k}$ for some $s\in\cS_U$ and $k\geq 0$. Because $\Ad(a)\cS^{2}\subset\cS$, there is some $C\geq 1$ such that, for all $a'\in\cS_A$, we have $\Ad(a)^{C}\Ad(a')\cS\subset\cS$. Let $K=C\ell(b)+k$, so that $\Ad(a)^{K-k}\Ad(b')\cS\subset\cS$ for any word $b'\in \cS_A^*$ with $\ell(b')\leq\ell(b)$, and let $s',s''\in\cS$ be given by $s'=\Ad(a)^{K-k}s$ and $s''=b^{-1}s'b$. We homotope as follows, liberally using Lemma \ref{lemma:big}.
$$\omega(u)b=a^{k}sa^{-k}b\leadsto a^{K}a^{k-K}sa^{K-k}a^{-K}b$$
$$\leadsto a^{K}s'ba^{-K}\leadsto a^{K}bb^{-1}s'ba^{-K}\leadsto
ba^{K}s''a^{-K}\leadsto b\omega(b^{-1}ub).$$

\paragraph{Filling $\omega$-triangles in $U$.} To show that $\omega(u_1)\omega(u_2)\omega(u_3)\leadsto\varepsilon$ for $u_1,u_2,u_3\in U$ such that $u_1u_2u_3=1$, write $\omega(u_i)$ as $a^{k_i}s_ia^{-k_i}$ for $i=1,2,3$, with $k_i\geq 0$ and $s_i\in\cS_U$. Let $K=k_1+k_2+k_3$ and let $s_i'=\Ad(a^{k_i-K})s_i\in\cS_U$. We homotope as follows.
$$\omega(u_1)\omega(u_2)\omega(u_3)
=(a^{k_1}s_1a^{-k_1})(a^{k_2}s_2a^{-k_2})(a^{k_3}s_3a^{-k_3})$$
$$\leadsto (a^{K}a^{k_1-K}s_1a^{K-k_1}a^{-K})(a^{K}a^{k_2-K}s_2a^{K-k_2}a^{-K})
(a^{K}a^{k_3-K}s_3a^{K-k_3}a^{-K})$$
$$\leadsto (a^{K}s_1'a^{-K})(a^{K}s_2'a^{-K})
(a^{K}s_3'a^{-K})
\leadsto a^{K}s_1's_2's_3'a^{-K}\leadsto a^Ka^{-K}\leadsto\varepsilon.$$
\end{proof}

\subsection{Filling freely trivial words.}
\label{subsection:free}
We now return to the case where the standard solvable group $G=U\rtimes A$ is not necessarily tame. Recall that the collection of conic subsets $\cC$ is finite. For $C\in\cC$, let $G_C$ be the tame group $U_C\rtimes A$, and let $\cS_{G_C}$ be a compact generating set for this group. Let $H=\Asterisk_{C\in\cC} G_C$, and let $\cS_H\subset H$ be the union of the $\cS_{G_C}$. There is a natural map from $H$ to $G$. Lemma \ref{lemma:freehomotopy} will show that if $w\in\cS_H^*$ represents the identity in $H$, then its image in $G$ admits a $O(\ell(w))$-Lipschitz filling. We will need the following auxiliary results first---the reader should probably skip directly to the proof of Lemma \ref{lemma:freehomotopy} to understand the point of these propositions. 

\begin{itemize}
\item Proposition \ref{proposition:free} shows that a word $w\in\cS_H^*$ representing the identity in $H$ may be reduced to the identity by repeated deletion of subwords $r_j\in\cS_{G_{C_j}}^*$ such that $r_j$ represents the identity in $G_{C_j}$.
\item Proposition \ref{proposition:rectangle} describes an appropriately Lipschitz rectangular homotopy between words obtained by these deletions.
\item Proposition \ref{proposition:trivial} describes a part of the homotopy given in Proposition \ref{proposition:rectangle}.
\end{itemize}

\begin{proposition}
\label{proposition:free}
Given a word $w\in\cS_H^*$ such that $w=_H 1$, there exists a natural number $n\leq \ell(w)$ and, for $j=0,\ldots,n$, words $a_j,r_j,b_j\in\cS_H^*$ with the following properties.
\begin{itemize}
\item $w=_{\cS_H^*} a_0r_0b_0$.
\item $a_n,r_n,b_n=\varepsilon$.
\item For all $j=0,\ldots, n$, there is some $C_j\in\cC$ such that $r_j\in\cS_{G_{C_j}}^*$ and $r_j=_{G_{C_j}} 1$.
\item $a_{j}b_{j}=_{\cS_H^*} a_{j+1}r_{j+1}b_{j+1}$ for $j=0,\ldots, n-1$.
\end{itemize}
\end{proposition}
\begin{proof}
Note that each element of $\cS_H$ lies in some $G_C$. If $w\neq \varepsilon$ represents the identity, then by the theory of free products, $w$ has a (nonempty) subword which is comprised of elements of some $\cS_{G_C}$ and represents the identity in $G_C$. Thus, we may write $w=_{\cS_H^*}a_0r_0b_0$ where $r_0\in \cS_{G_{C_0}}^*$ as desired. Applying this argument recursively, we obtain $a_j,r_j,b_j$ as desired.
\end{proof}

\begin{proposition}
\label{proposition:trivial}
Given a natural number $k$ and $w\in \cS_H^*$, there exists a $O(k+\ell(w))$ Lipschitz map $f:[0,1]\times\left[0,\frac{k}{k+\ell(w)}\right]\To G$ with the following properties.
\begin{itemize}
\item Along the bottom, $f$ is given by $1^k w$, meaning $f(t,0)=\gamma_{1^k w}(t)$ for $0\leq t\leq 1$.
\item Along the top, $f$ is given by $w 1^k$, meaning $f\left(t,\frac{k}{k+\ell(w)}\right)=\gamma_{w 1^k}(t)$ for $0\leq t \leq 1$.
\item $f$ is constant on the sides, so $f(0,s)$ and $f(1,s)$ do not depend on $s$.
\end{itemize}
\end{proposition}

\begin{proof}
Let $\gamma:\RR\To G$ be given as follows. $\gamma(t)=1_G$ for $t\leq 0$, $\gamma(t)=\gamma_{1^k w}(t)$ for $0<t\leq 1$, and $\gamma(t)=\gamma_{1^k w}(1)$ for $t>1$. Observe that $\Lip(\gamma)=O(k+\ell(w))$ Then we may take
$$f(t,s)=\gamma(t+s)$$
as our desired filling.
\end{proof}

\begin{proposition}
\label{proposition:rectangle}
Given $a,b\in \cS_H^*$, $r\in\cS_{G_C}^*$ a relation in $G_C$ for some $C\in\cC$, and any natural number $k$, let $\ell=\ell(arb1^k)$ and $h=\frac{\ell(r)}{\ell}$. There exists a $O(\ell)$-Lipschitz map $f:[0,1]\times\left[0,h\right]\To G$ with the following properties.
\begin{itemize}
\item Along the bottom, $f$ is given by $arb1^k$, meaning $f(t,0)=\gamma_{arb1^k}(t)$ for $t\in[0,1].$
\item Along the top, $f$ is given by $ab1^{k+\ell(r)}$, meaning $f\left(t,h\right)=\gamma_{ab1^{k+\ell(r)}}(t)$ for $t\in[0,1]$.
\item $f$ is constant on the sides, so $f(0,s)$ and $f(1,s)$ do not depend on $s$.
\end{itemize}
\end{proposition}

\begin{proof}
(See the right hand side of Figure \ref{figure:freefilling}.) Subdivide the rectangle into $[0,1]\times[0,h/2]$ and $[0,1]\times[h/2,1]$. Define $f$ to be $a1^{\ell(r)}b1^k$ on $[0,1]\times\{h/2\}$---i.e.,
$$f(t,h/2)=\gamma_{a1^{\ell(r)}b1^k}(t).$$

First, we extend $f$ over the top rectangle $[0,1]\times [h/2,1]$. For
$$(t,s)\in \left[0,\frac{\ell(a)}{\ell}\right]\times[h/2,1]
\cup \left[1-\frac{k}{\ell},1\right]\times[h/2,1]$$
we have that $\gamma_{a1^{r}b1^k}(t)=\gamma_{ab1^{\ell(r)+k}}(t)$, so we can just set $f$ to be constant vertically---i.e., we define $f(t,s)=\gamma_{ab1^{\ell(r)+k}}(t)$ for these $(t,s)$. To extend $f$ to
$$(t,s)\in\left[\frac{\ell(a)}{\ell},1-\frac{k}{\ell}\right]\times[h/2,1],$$
we simply apply Proposition \ref{proposition:trivial}. Thus, we have given a $O(\ell)$-Lipschitz extension of $f$ over the top rectangle.

Now we extend over the bottom rectangle $[0,1]\times [0,h/2]$. For
$$(t,s)\in \left[0,\frac{\ell(a)}{\ell}\right]\times[0,h/2]
\cup \left[\frac{\ell(ar)}{\ell},1\right]\times[0,h/2]$$
define $f(t,s)=\gamma_{arb1^k}(t)$. Finally, we apply Proposition \ref{proposition:tamelip1c} to extend $f$ over $\left[\frac{\ell(a)}{\ell},\frac{\ell(ar)}{\ell}\right]\times[0,h/2]$, since this is equivalent to filling $r$.
\end{proof}

\begin{lemma}
\label{lemma:freehomotopy}
Recalling the notation introduced at the start of \S\ref{subsection:free}, we have $w\leadsto\varepsilon$ (in $G$) for all $w\in\cS_H^*$ such that $w=_H 1$
\end{lemma}

\begin{proof} (See Figure \ref{figure:freefilling}.)
Let $w\in \cS_H^*$ be a relation in $H$ and take a sequence of words $a_j,r_j,b_j$, $j=0,\ldots,n$ as in Proposition \ref{proposition:free}. We will define a $O(\ell(w))$-Lipschitz filling $f:[0,1]\times[0,1]\To G$ of $\gamma_w$ as follows. Let $\ell_k=\sum_{j<k}\ell(r_j)$, so that $\ell_0=0$ and $\ell_n=\ell(w)$, and subdivide $[0,1]\times [0,1]$ into rectangles $[0,1]\times[\ell_j,\ell_{j+1}]$ for $j=0,\ldots,n-1$. Set
$$f(t,\ell_j)=\gamma_{a_jr_jb_j1^{\ell_j}}(t),$$
noting that $\ell(a_jr_jb_j1^{\ell_j})=\ell(w)$.

Proposition \ref{proposition:rectangle} now shows that $f$ may be extended over each rectangle $[0,1]\times[\ell_j,\ell_{j+1}]$ with Lipschitz constant $O(\ell(w))$.
\end{proof}

\begin{figure}[t]
\labellist
\small\hair 2pt


\pinlabel $a_0$ at 63 11
\pinlabel $r_0$ at 150 11
\pinlabel $b_0$ at 230 11

\pinlabel $a_1$ at 32 90
\pinlabel $r_1$ at 60 90
\pinlabel $b_1$ at 131 90
\pinlabel $1^{\ell_1}$ at 223 90

\pinlabel $a_2$ at 62 122
\pinlabel $r_2$ at 127 122
\pinlabel $b_2$ at 152 122
\pinlabel $1^{\ell_2}$ at 212 122

\pinlabel $\ldots$ at 140 165

\pinlabel $r_{n-1}$ at 47 204
\pinlabel $1^{\ell_{n-1}}$ at 169 204

\pinlabel $1^{\ell_n}$ at 140 275

\pinlabel $a$ at 331 80
\pinlabel $r$ at 408 80
\pinlabel $b$ at 469 80
\pinlabel $1^{k}$ at 537 80

\pinlabel $1^{r}$ at 408 128
\pinlabel $b$ at 469 128

\pinlabel $a$ at 331 192
\pinlabel $b$ at 372 192

\pinlabel $1^{k+\ell(r)}$ at 490 214

\pinlabel $\text{Prop. \ref{proposition:tamelip1c}}$ at 408 110
\pinlabel $\text{Proposition \ref{proposition:trivial}}$ at 420 158

\endlabellist

\centering
\centerline{\psfig{file=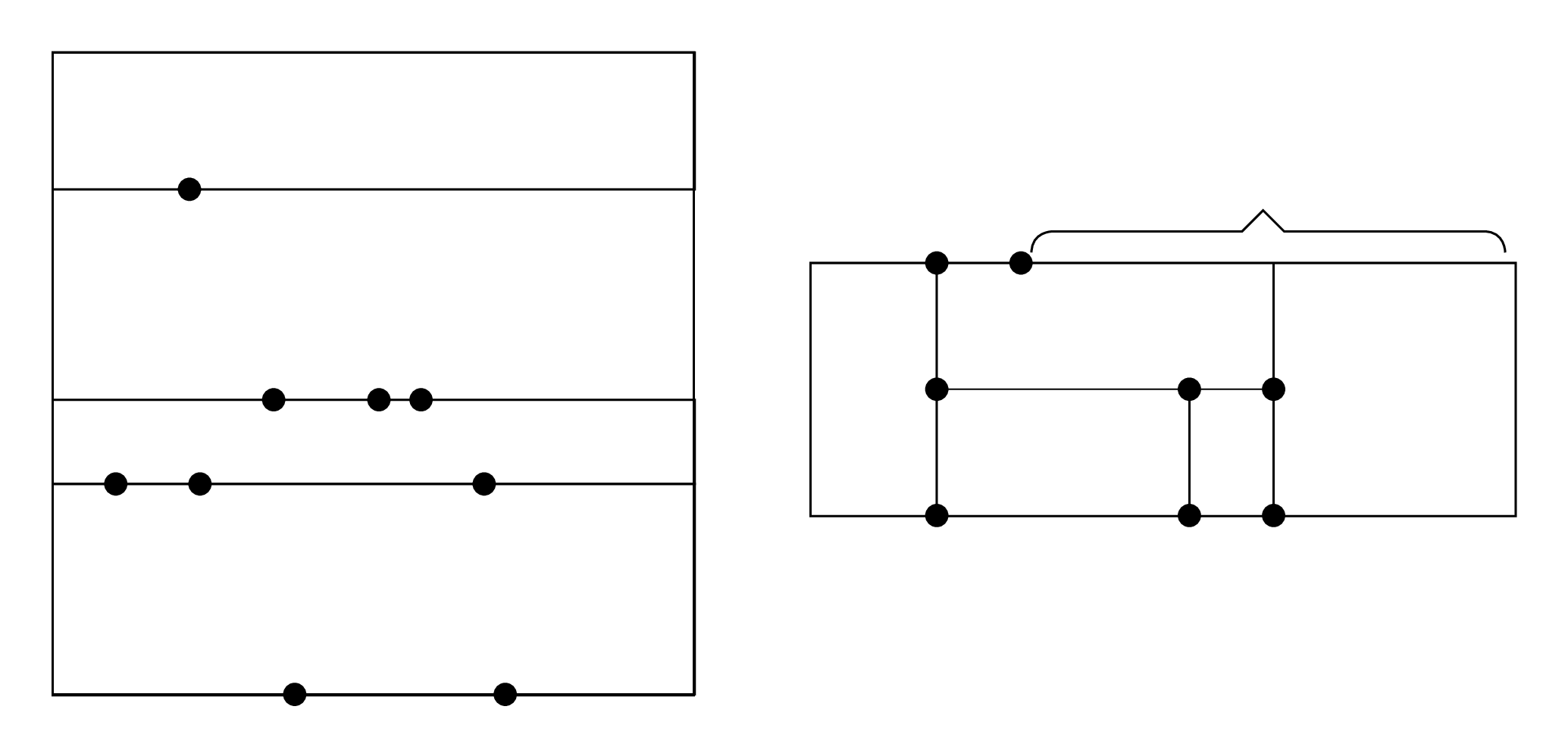,scale=70}}
\caption{The figure on the left depicts our strategy for filling of the freely trivial word $w=a_0 r_0 b_0$, where the $a_j,r_j,b_j$ are as in Proposition \ref{proposition:free}. The figure on the right depicts the proof of Proposition 
\ref{proposition:rectangle} which allows us to fill in each rectangle in the left hand figure.}
\label{figure:freefilling}
\end{figure}
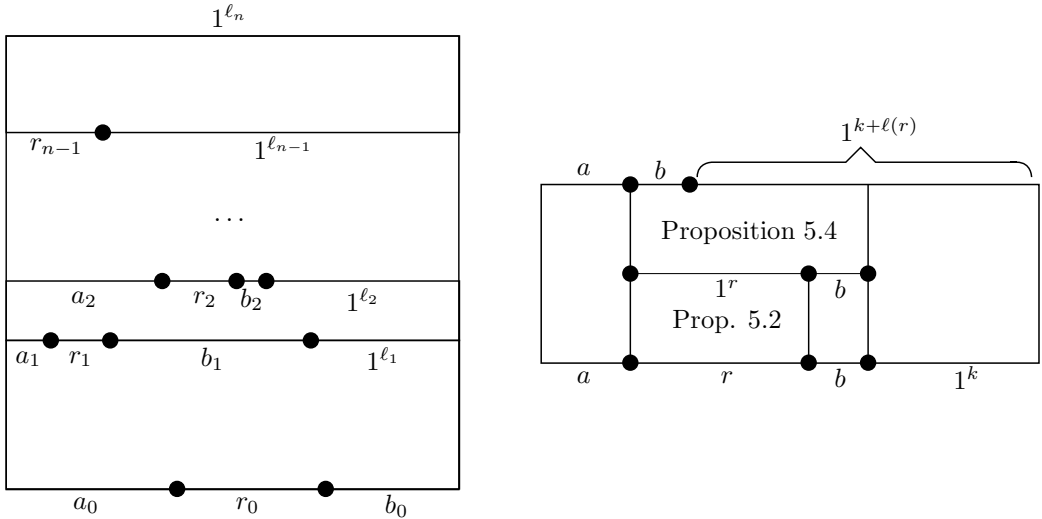

\paragraph{Distortion.} We now see that if $G$ is standard solvable, then elements of $U$ may be expressed much more efficiently in the generators of $G$ than in the generators of $U$.

\begin{proposition}
\label{proposition:shortcuts}
Suppose $G=U\rtimes A$ is standard solvable, $\cS$ is a compact generating set for $G$, and $\cS_U$ is a compact generating set for $U$.

There exists $C>1$ such that if $u\in U\setminus\{1_U\}$, then
$$\frac{1}{C}\log(1+|u|_{\cS_U})\leq |u|_\cS\leq C \log(1+|u|_{\cS_U}).$$
\end{proposition}

\begin{proof}
This follows, with some effort, from \cite[Proposition 6.B.2]{ct}.
\end{proof}

\subsection{The multiamalgam.}
In this subsection, we will define the multiamalgam $\hat{G}$ of a standard solvable group $G=U\rtimes A$ (first introduced by Abels \cite{ab}), and quote a key theorem of Cornulier and Tessera, which states that certain conditions under which $G\cong\hat{G}$---this means that $G$ is put together from its standard tame subgroups in a nice way, which will eventually let us build fillings in $G$ from fillings in standard tame subgroups. In order to state this theorem in the proper generality, we must briefly discuss the theory of unipotent groups.

\paragraph{Unipotent groups.} For a commutative $\RR$-algebra $\cP$, and a real unipotent group $U$ (i.e., a closed group of upper triangular real matrices with diagonal entries equal to $1$), the theory of algebraic groups allows us to define a group $U(\cP)$ \cite[\S 1.4]{borel}. In particular, if $U\subset \GL(n;\RR)$ consists of all upper triangular matrices with diagonal entries equal to $1$, then $U(\cP)$ consists of all upper triangular $n\times n$ matrices over $\cP$ with diagonal entries equal to $1$---such matrices are certainly invertible, having determinant equal to $1$. Suppose $\cP=\RR^Y$, so that $\cP$ consists of all functions $f:Y\To\RR$. Then there is an obvious bijection $U(\cP)\leftrightarrow U^Y$, and for $y\in Y$ and $\tilde{u}\in U(\cP)$ we may speak of $\tilde{u}(y)\in U$.

\begin{definition}(See \cite[\S 10.B]{ct}).
Let $G=U\rtimes A$ be real standard solvable. The multiamalgam $\hat{G}$ of the standard tame subgroups $G_C$ is defined by
$$\hat{G}=\Asterisk_{C\in\cC} G_C/\la\la R_G\ra \ra$$
where $R_G=\{i_C(u)^{-1}i_{C'}(u):u\in G_C\cap G_{C'}\}$ and $i_C$ denotes the inclusion of $G_C$ in the direct product.

Similarly, the multiamalgam $\hat{U}$ is defined by
$$\hat{U}=\Asterisk_{C\in\cC} U_C/\la\la R_U\ra \ra$$
where $R_U=\{i_C(u)^{-1}i_{C'}(u):u\in U_C\cap U_{C'}\}$ and $i_C$ denotes the inclusion of $U_C$ in the direct product.

For any commutative $\RR$-algebra $\cP$, we define $\hat{U}(\cP)$ and $\hat{G}(\cP)$ similarly, where $G_C(\cP)$ is understood to be $U_C(\cP)\rtimes A$.
\end{definition}

Of course, $\hat{U}\rtimes A\cong \hat{G}$. Recall that $G$ admits the SOL obstruction if it surjects onto a group of SOL type. Cornulier and Tessera give conditions under which $\hat{U}$ is isomorphic to $U$.

\begin{theorem}
\label{theorem:multiamalgam}
Let $G=U\rtimes A$ be a standard solvable real Lie group.  If $H_2(\mfu)_0=0$, $\Kill_2(\mfu)_0=0$, and $G$ does not admit the SOL obstruction then $\hat{U}(\cP)\cong U(\cP)$ for all commutative $\RR$-algebras $\cP$.
\end{theorem}

\begin{proof}
This follows from Corollary 9.D.4 of \cite{ct}. The 2-tameness hypotheses of the corollary is satisfied because of Proposition 4.C.3 of \cite{ct}.
\end{proof}

\section{Proof of the main theorem.}
\label{section:main}
The rest of this paper is devoted to the proof of the following theorem.

\begin{theorem}
\label{theorem:MAIN}
Let $G=U\rtimes A$ where $U$ and $A$ are contractible real Lie groups, $A$ is abelian, and $U$ is a real unipotent group (i.e., a closed group of strictly upper triangular real matrices.)

If $G$ is standard solvable and does not surject onto a group of Sol type, and $H_2(\mfu)_0$ and $\Kill(\mfu)_0$ are trivial, then $G$ is Lipschitz $1$-connected.
\end{theorem}

\begin{proof}
Lemma \ref{lemma:omega} will show that there exists a generating set $\cS$ for $G$ and normal form $\omega:G\To\cS^*$ with certain properties.

Lemma \ref{lemma:omegatri} will show that if $\omega$ has these properties, then $\Delta\leadsto\varepsilon$ for $\omega$-triangles $\Delta$. By Lemma \ref{lemma:gromov}, this will suffice to prove the theorem.
\end{proof}

\subsection{Defining $\omega$.}
\label{subsection:normalform}
\paragraph{Standing assumptions.} Throughout the rest of this paper, we assume that $G$ satisfies the hypotheses of the theorem. That is, $G=U\rtimes A$ is a standard solvable group such that $H_2(\mfu)_0=0$, $\Kill_2(\mfu)_0=0$, and $G$ does not surject onto a group of Sol type.

\paragraph{Notation.} Let $H=\Asterisk_{C\in\cC}G_C$ and $H_U[\cP]=\Asterisk_{C\in\cC}U_C(\cP)$ for any commutative $\RR$ algebra $\cP$, and let $i_C:U_C(\cP)\To H_U[\cP]$ denote inclusion. We will write $H_U$ for $H_U[\RR]$. Let $\cS_A$ be a compact generating set for $A$. For $C\in\cC$, let $\cS_C$ be a compact generating set for $U_C$. Let $\cS_U=\bigcup_{C\in\cC}\cS_C$---by Theorem \ref{theorem:multiamalgam} this is a compact generating set for $U$. Let $\cS=\cS_A\cup\cS_U$, this is a compact generating set for $G$. Let $\cS_H$ be a generating set for $H$ which is equal to the union of compact generating sets for $G_C$ as $C$ ranges over $\cC$, and let $\cS_{H_U}\subset H_U$ be the union of the $\cS_C$---this is a generating set for $H_U$. Given $C\in\cC$ and $x\in U_C$, let $\olx\in(\cS_A\cup\cS_C)^*$ be a minimal length word representing $x$. Let $\phi_A:G\To A$ be projection. Let the set theoretic map $\phi_U:G\To U$ be defined by $\phi_U(g)=\phi_A(g)^{-1}g$, so that $g=\phi_A(g)\phi_U(g)$.

\begin{lemma}
\label{lemma:omega}
Under our standing assumptions, there exists a finite sequence $C_1\ldots C_k$ of conic subsets and a normal form $\omega:G\To\cS^*$ such that $\omega$ has the following properties.
\begin{itemize}
\item For any $g\in G$, $\omega(g)=\omega(\phi_A(g))\omega(\phi_U(g))$.
\item For $a\in A$, $\omega(a)\in\cS_A^*$ is a minimal length word representing $a$.
\item For $u\in U$, $\omega(u)$ has the form $\overline{x_1}\ldots\overline{x_k}$, where $x_i\in U_{C_i}$. 
\end{itemize}
\end{lemma}

\begin{proof}
This follows from Proposition 6.B.2 of \cite{ct}, but we will now give a different proof in order to introduce a trick that will be used later.

\paragraph{The Cornulier-Tessera trick.} For a set $Y$, define the commutative $\RR$-algebra $\cP_Y$ as the collection of all functions $f:Y\times[1,\infty)\To\RR$ such that there is some $\beta\in\NN$ with $|f(y,t)|<(1+t)^{\beta}$. Note that an element of $U(\cP_Y)$ may be identified with a function from $Y\times[1,\infty)$ to $U$. Alternatively, one may think of an element of $U(\cP_Y)$ as a family of functions $[1,\infty)\To U$, indexed by $Y$ with matrix coefficients uniformly bounded by some polynomial $(1+t)^\beta$.

Choose $Y$ to have at least continuum cardinality, and let $\tilde{g}\in U(\cP_Y)$ be such that for every $g\in U$, there is some $y\in Y$ and $t=O(|g|_{\cS_U})$ such that $\tilde{g}(y,t)=g$. It is certainly possible to do this---for instance, one might take $Y=U$ and set $\tilde{g}(y,t)$ to be $1_U$ for $t<|y|_{\cS_U}$ and $y$ for $t\geq|y|_{\cS_U}$. In claiming that $\tilg\in U(\cP_Y)$, we have used the fact that the matrix coefficients of $g\in U$ are at most polynomial in $|g|_{\cS_U}$.

Since, by Theorem \ref{theorem:multiamalgam}, $U(\cP_Y)$ is generated by the union of the $U_C(\cP_Y)$, we may write $\tilg=\tilx_1\ldots \tilx_k$ where each $\tilx_i$ is an element of some $U_{C_i}(\cP_Y)$. Observe that there exists some $\alpha\in\NN$ such that $|\tilx_i(y,t)|_{\cS_{C_i}}\leq(1+t)^{\alpha}$ for all $i=1,\ldots,k$, $y\in Y$, and $t\geq 1$.

Now, let $g$ be an element of $U$. By definition of $\tilg$, there exist $y\in Y$ and $t=O(|g|_{\cS_U})$ such that $\tilg(y,t)=g$.  For $i=1,\ldots,k$, let $x_i=\tilx_i(y,t)$. We have that
$$g=\tilg(y,t)=\tilx_1(y,t)\ldots\tilx_k(y,t)=x_1\ldots x_k,$$
and $x_i\in U_{C_i}(\cP_Y)$ with $|x_i|_{\cS_{C_i}}=O(t^\alpha)=O(|g|_{\cS_U}^\alpha)$. Take $\omega(g)$ to be $\overline{x_1}\ldots\overline{x_k}$, and note that $|\omega(g)|_\cS=O(\log|g|_{\cS_U})$ because $|\overline{x_i}|_{\cS_A\cup\cS_{C_i}}=O(\log|x_i|_{\cS_{C_i}})=O(\log|g|_{\cS_U})$ by \ref{proposition:shortcuts}.

We have thus defined $\omega$ on elements of $U$, with the desired properties. Define $\omega$ on $A$ by taking $\omega(a)$ to be the shortest word in $\cS_A^*$ representing $a\in A$. Extend $\omega$ to all of $g$ by setting $\omega(g)=\omega(\phi_A(g))\omega(\phi_U(g))$. We must show that $\omega$ is a normal form---i.e., that, for $g\in G$, $\ell(\omega(g))=O(|g|_\cS)$.

We have $\ell(\omega(\phi_A(g)))=|\phi_A(g)|_\cS=O(|g|_\cS)$, so it suffices to show that $\ell(\omega(\phi_U(g)))=O(|g|_\cS)$. If $w\in\cS^*$ is a minimal length word representing some $g\in G$, note that $\omega(\phi_A(g))^{-1}w$ represents $\phi_U(g)$. By \ref{proposition:shortcuts}, there is some constant $C>1$ not depending on $g$ such that
$$|\omega(\phi_A(g))^{-1}w|_\cS\geq\frac{1}{C}\log|\phi_U(g)|_{\cS_U}.$$
Thus, since $|\omega(\phi_U(g))|_\cS=O(\log|\phi_U(g)|_{\cS_U})$, we have that
$$\omega(\phi_U(g))=O(|\phi_A(g)|_\cS+|w|_\cS)=O(|g|_\cS)$$
as desired.
\end{proof}

\subsection{Filling $\omega$-triangles.}
\label{subsection:normalformtriangles}
We wish to show that we may fill $\omega$ triangles, where $\omega$ is a normal form produced by Lemma \ref{lemma:omega}. Proposition \ref{proposition:conjugation} will allow us to homotope $\omega$-triangles into relations of the form $\overline{x_i}\ldots\overline{x_K}$ where each $\overline{x_i}$ is a word in $\cS_A\cup\cS_{C_i}$ efficiently representing an element $x_i$ of $U_{C_i}$, where $C_1,\ldots,C_K$ is some fixed sequence of conic subsets. In order to fill such relations, recall from Theorem \ref{theorem:multiamalgam} that under our standing assumptions, $U(\cP)\cong \hU(\cP)$ for any commutative $\RR$-algebra $\cP$. Consequently, the kernel of
$H_U[\cP] \To U(\cP)$
is normally generated by elements of the form
$$i_C(u)^{-1}i_{C'}(u)\quad\quad\quad\quad [u\in U_C(\cP)\cap U_{C'}(\cP).]$$
In order to fill $\overline{x_1}\ldots\overline{x_K}$ in Corollary \ref{corollary:factoring}, we will need to factor $x_1\ldots x_K$ in the free product $H_U$ as a product of a bounded number of elements of the form $g^{-1}i_C(u)^{-1}i_{C'}(u)g$, where each $g\in H_U$ is a product of a bounded number of elements living in some factors $U_C$, with $|g|_{\cS_U}$ and $|u|_{\cS_U}$ controlled by some polynomial of $\sum_{j=1}^K|x_j|_{\cS_{C_j}}$.

\begin{lemma}
\label{lemma:factoring}
(See \cite[Lemma 7.B.1]{ct}). Suppose that our standing assumptions are satisfied. Given a sequence of conical subsets $C_1,\ldots C_K$, there exist natural numbers $N,\mu,\beta$ such that for any sequence $x_i\in U_{C_i}$ with $x_1x_2\ldots x_K =_U 1_U$ there is an equality of the form
$$x_1\ldots x_K =_{H_U} (g_1 r_1 g_1^{-1})\ldots(g_N r_N g_N^{-1})$$
satisfying
\begin{itemize}
\item $g_j=_{H_U} g_{j1}\ldots g_{j\mu}$ where $g_{jk}\in U_{C_{jk}}$ for some $C_{jk}\in \cC$.
\item Each $r_j$ is of the form $i_{C'_j}(u_j)i_{C''_j}(u_j)^{-1}$ for some conical subsets $C'_j,C''_j$ and some $u_j\in U_{C'_j} \cap U_{C''_j}$.
\item $|g_{jk}|_{\cS_{C_{jk}}}=O(\ell^\beta),|u_j|_{\cS_{C'_j}}=O(\ell^\beta),$ and $|u_j|_{\cS_{C''_j}}=O(\ell^\beta)$ where $\ell=1+\sum_{i=1}^K |x_i|.$
\end{itemize}
\end{lemma}

\begin{proof}
This is a special case of \cite[Lemma 7.B.1]{ct}, but we will reprise most of the details here. We will use the same Cornulier and Tessera trick we used to prove Lemma \ref{lemma:omega}.
Take $\cP_Y$ as in the proof of Lemma \ref{lemma:omega}. Recall that $\tilx\in U_C(\cP_Y)$ may be thought of as a function from $Y\times[0,\infty)$ to $U_C$. Let $Y$ be a set with at least continuum cardinality, so that there exists $(\tilx_1,\ldots,\tilx_K)\in U_{C_1}(\cP_Y)\times
\ldots\times U_{C_K}(\cP_Y)$ which has the following strong surjectivity property: for any $(x_1\ldots x_K)\in U_{C_1}(\cP_Y)\times\ldots\times U_{C_K}(\cP_Y)$ with $x_1\ldots x_K=_U 1$ there exists $y\in Y$ and $t=O(|x_1|_{\cS_{C_1}}+\ldots+|x_K|_{\cS_{C_K}})$ with $\tilx_i(y,t)=x_i$.

By Theorem \ref{theorem:multiamalgam}, we know that there is some equality of the form
$$\tilx_1\ldots\tilx_K=_{H_U[\cP_Y]}
(\tilg_1^{-1}\tilr_1\tilg_1)\ldots(\tilg_N^{-1}\tilr_N\tilg_N)$$
where $\tilg\in\Asterisk U_C(\cP_Y)$ and each $\tilr_j$ has the form $i_{C'_j}(\tilu_j)^{-1}i_{C^{\prime\prime}_{j}}(\tilu_j)$ for some conical subsets $C'_j,C''_j$ and $\tilu_j\in U_{C'_j}(\cP_Y)\cap U_{C''_j}(\cP_Y)$. Since the $U_C(\cP_Y)$ generate $H_U[\cP_Y]$, there must be some $\mu$ such that for all $j=1,\ldots,N$,
$$\tilg_j=\tilg_{j1}\ldots\tilg_{j\mu},$$
where each $\tilg_{jk}$ lives in $U_{C_{jk}}$ for some $C_{jk}\in\cC$. Note that by definition of $\cP_Y$, there is some $\beta$ such that all $|\tilg_{jk}(y,t)|_{\cS_U}$ and $|\tilu_j(y,t)|_{\cS_{C'_j}}$ are $O(t^\beta)$.

Given, $x_1,\ldots,x_K\in U_{C_1}\times\ldots\times U_{C_K}$, let $\ell=\sum_{i=1}^K|x_i|_{\cS_{C_i}}$ and choose $y\in Y$ and $t=O(\ell)$ such that $\tilx_i(y,t)=x_i$ for $i=1,\ldots,K$. For $j=1,\ldots,N$ and $k=1,\ldots,\mu$, let $g_j=\tilg_j(y,t)$, $g_{jk}=\tilg_{jk}(y,t)$, $u_j=\tilu_j(y,t)$, and $r_j=i_{C'}(u_j)^{-1}i_{C''}(u_j)$. It follows that
$$x_1\ldots x_K =_{H_U} (g_1 r_1 g_1^{-1})\ldots(g_N r_N g_N^{-1}),$$
and the $g_j$ and $r_j$ satisfy the desired conditions.
\end{proof}

\begin{corollary}
\label{corollary:factoring}
Suppose that our standing assumptions are satisfied. Given a sequence of conical subsets $C_1,\ldots C_K$,  we have, in $G$, that
$$\overline{x_1}\ldots \overline{x_K} \leadsto \varepsilon.$$
for any sequence $x_i\in U_{C_i}$ with $x_1x_2\ldots x_K =_U 1_U$.
\end{corollary}

\begin{proof}
By Lemma \ref{lemma:factoring}, we have (for $\beta,N,\mu$ independent of the $x_j$),
$$x_1\ldots x_K =_{H_U} (g_1 r_1 g_1^{-1})\ldots(g_N r_N g_N^{-1})$$
where $g_j=_{H_U}g_{j1}\ldots g_{j\mu}$ for $g_{jk}\in U_{C_{jk}}$, each $r_j$ is of the form $i_C'(u_j)^{-1}i_{C''}(u_j)$ for some conical subsets $C'_j,C''_j$ and some $u_j\in U_{C'_j} \cap U_{C''_j}$, and $|g_{jk}|_{\cS_{C_{jk}}},|u_j|_{\cS_{C'_j}}=O(\ell^\beta)$ where $\ell=1+\sum_{i=1}^K |x_i|.$

For each $j=1,\ldots,N$, let $\overline{g_j}\in\cS^*$ be $\overline{g_{j1}}\ldots\overline{g_{j\mu}}$. Let $\overline{r_j}\in\cS^*$ be equal to $(\olu'_{j})^{-1}\olu''_j$, where $\olu'_j\in \cS_A\cup\cS_{C'_j}$ is a minimal length word representing $u_j$ in $G_{C'_j}$ and $\olu''_j\in \cS_A\cup\cS_{C''_j}$ is a minimal length word representing $u_j$ in $G_{C''_j}$---this implies that $\overline{r_j}$ represents $r_j$ in $H$.

First we show that $\overline{r_j}\leadsto \varepsilon$. Note that $C'_j\cap C''_j$ is itself a conic subset, so there is some $\overline{u_j}\in (\cS_A\cup \cS_{C'_j\cap C''_j})^*$ which represents $u_j$ in $G_{C'_j\cap C''_j}$ with $\ell(\overline{u_j})=O(\ell(\overline{r_j}))$. Thus we have (by Proposition \ref{proposition:tamelip1c}),

$$\overline{r_j}=(\olu'_j)^{-1}\olu''_j\leadsto\overline{u_j}^{-1}\overline{u_j}\leadsto\varepsilon.$$

Next, observe that
$$\ell(\overline{g_j})=\sum_{k=1}^\mu\ell(\overline{g_{jk}})=\sum_{k=1}^\mu O(\log|g_{jk}|_{\cS_{C_{jk}}})$$
$$=O\left(\beta\log\left(1+\sum_{i=1}^K|x_i|\right)\right)=O(\ell(\overline{x_1}\ldots\overline{x_K})),$$
by Proposition \ref{proposition:shortcuts}, and $\ell(\overline{r_i})=O(\ell(\overline{x_1}\ldots\overline{x_K}))$ similarly.

Because
$$\overline{x_1}\ldots \overline{x_K} =_{H} (\overline{g_1} \overline{r_1} \overline{g_1}^{-1})\ldots(\overline{g_N} \overline{r_N} \overline{g_N}^{-1}),$$
and
$$\ell((\overline{g_1} \overline{r_1} \overline{g_1}^{-1})\ldots(\overline{g_N} \overline{r_N} \overline{g_N}^{-1}))=O(\ell(\overline{x_1}\ldots\overline{x_K})),$$
we have by Lemma \ref{lemma:big} and Lemma \ref{lemma:freehomotopy} that
$$\overline{x_1}\ldots \overline{x_K}\leadsto(\overline{g_1} \overline{r_1} \overline{g_1}^{-1})\ldots(\overline{g_N} \overline{r_N} \overline{g_N}^{-1})$$
$$\leadsto(\overline{g_1}\overline{g_1}^{-1})\ldots(\overline{g_N}\overline{g_N}^{-1})\leadsto\varepsilon$$
because $\overline{r_j}\leadsto\varepsilon$ as noted above.
\end{proof}

We need one more proposition before we can fill $\omega$-triangles. Given $a,x\in G$, let $^ax$ denote $axa^{-1}$.

\begin{proposition}
\label{proposition:conjugation}
Fix a sequence of conical subsets $C_1,\ldots,C_K\in\cC$. We have that
$$\omega(a)\overline{x_1}\ldots\overline{x_K}\leadsto\overline{^ax_1}\ldots\overline{^ax_2}\omega(a)$$
for all $a\in A$ and $x_i\in U_{C_i}$.
\end{proposition}

\begin{proof}
$$\omega(a)\overline{x_1}\ldots\overline{x_K}\leadsto
(\omega(a)\overline{x_1}\omega(a)^{-1})(\omega(a)\overline{x_2}\omega(a)^{-1})\ldots(\omega(a)\overline{x_K}\omega(a)^{-1})\omega(a)$$
$$\leadsto
\overline{^ax_1}\ldots\overline{^ax_K}\omega(a),
$$
by Proposition \ref{proposition:tamelip1c}.
\end{proof}

We now conclude the proof of our main theorem by showing that we may fill $\omega$-triangles.

\begin{lemma}
\label{lemma:omegatri}
Under our standing assumptions, if $g_1,g_2,g_3\in G$ with $g_1g_2g_3=_G 1_G$, we have
$$\omega(g_1)\omega(g_2)\omega(g_3)\leadsto \varepsilon.$$
\end{lemma}

\begin{proof}
Recall that $\omega(g)$ has the form $\omega(a)\overline{x_1}\ldots\overline{x_k}$ where $\overline{u}_i\in\cS_A\cup\cS_{C_i}$ for some fixed sequence $C_1,\ldots C_k$ of conical subsets. Let $g_1,g_2,g_3\in G$ with $g_1g_2g_3=1_G$, and let $a_i=\phi_A(g_i)$ for $i=1,2,3$. Let $\omega(\phi_U(g_1))=\overline{x_1}\ldots\overline{x_k}$, $\omega(\phi_U(g_2))=\overline{x'_1}
\ldots\overline{x'_k}$ and $\omega(\phi_U(g_3))=\overline{x''_1}\ldots\overline{x''_k}$. Expanding and applying Proposition \ref{proposition:conjugation} repeatedly, we slide the $a$-words to the right to see that
$$\omega(g_1)\omega(g_2)\omega(g_3)
=\omega(a_1)\overline{x_1}\ldots\overline{x_k}
\omega(a_2)\overline{x'_1}\ldots\overline{x'_k}
\omega(a_3)\overline{x''_1}\ldots\overline{x''_k}$$
$$\leadsto
\overline{^{a_1}x_1}\ldots\overline{^{a_1}x_k}
\overline{^{a_2a_1}x'}_1\ldots\overline{^{a_2a_1}x'_k}
\overline{^{a_3a_2a_1}x''}_1\ldots\overline{^{a_3a_2a_1}x''_k}
\omega(a_1)\omega(a_2)\omega(a_3)$$
$$\leadsto
\overline{^{a_1}x_1}\ldots\overline{^{a_1}x_k}
\overline{^{a_2a_1}x'}_1\ldots\overline{^{a_2a_1}x'_k}
\overline{^{a_3a_2a_1}x''}_1\ldots\overline{^{a_3a_2a_1}x''_k}.
$$
This resulting word admits a Lipschitz filling by Corollary \ref{corollary:factoring}.
\end{proof}

\bibliographystyle{plain}
\bibliography{bibliography}

\noindent
David Bruce Cohen\\
Department of Mathematics\\
University of Chicago\\
5734 S. University Avenue,\\
Room 208C\\
Chicago, Illinois 60637\\
E-mail: {\tt davidbrucecohen@gmail.com}\\

\end{document}